\documentclass[11pt]{article}

\usepackage[greek,english]{babel}
\usepackage[utf8]{inputenc}

\usepackage{amsmath,amsfonts,amssymb,amsthm}
\usepackage{fullpage, ifthen, color, graphicx, hyperref, enumerate, relsize, verbatim, subfigure, enumitem}

\usepackage{amsthm}
\usepackage{mathrsfs}
\usepackage{comment}

\newtheorem{theorem}{Theorem}
\newtheorem{definition}[theorem]{Definition}
\newtheorem{lemma}[theorem]{Lemma}
\newtheorem{corollary}[theorem]{Corollary}
\newtheorem{proposition}[theorem]{Proposition}

\newcommand{\vv}[1]{\overrightarrow{#1}}

\makeatletter
\providecommand*{\cupdot}{%
  \mathbin{%
    \mathpalette\@cupdot{}%
  }%
}
\newcommand*{\@cupdot}[2]{%
  \ooalign{%
    $\m@th#1\cup$\cr
    \hidewidth$\m@th#1\cdot$\hidewidth
  }%
}
\makeatother
\allowdisplaybreaks

\title{Structure and enumeration of $K_4$-minor-free\\ links and link-diagrams\thanks{J.Ru\'e was partially supported by the Spanish MICINN project MTM2017-82166-P and the Mar\'ia de Maetzu research grant MDM-2014-0445 (BGSMath). D.M.Thilikos was supported by the projects ``ESIGMA" (ANR-17-CE23-0010) and ``DEMOGRAPH'' (ANR-17-CE23-0010). V.Velona was supported by the Spanish Ministry of Economy and Competitiveness, Grant MTM2015-67304-P and FEDER, EU. Partially supported by the MINECO project MTM2017-82166-P and the Mar\'ia de Maetzu research grant MDM-2014-0445 (BGSMath).}}

\author{
Juanjo Rué\thanks{Departament de Matemàtiques, Universitat Politècnica de Catalunya and Barcelona Graduate School of Mathematics, Barcelona, Spain. Email: {\tt juan.jose.rue@upc.edu}}
\and
Dimitrios M. Thilikos\thanks{LIRMM, Univ Montpellier, CNRS, Montpellier, France. Email: {\tt sedthilk@thilikos.info}}
\and
Vasiliki Velona\thanks{Department of Economics, Universitat Pompeu Fabra, and Departament de Matemàtiques, Universitat Politècnica de Catalunya, and Barcelona Graduate School of Mathematics, Spain.   Email: {\tt vasiliki.velona@upf.edu}}
}

\begin{document}

\maketitle

\begin{abstract}
 \noindent  We study the class  $\mathcal{L}$ of link-types that admit a $K_4$-minor-free diagram, i.e.,  they can be projected on the plane so that the resulting graph does not contain any subdivision of  $K_{4}$. We prove that $\mathcal{L}$ is  the closure of a subclass of torus links under the operation of connected sum. Using this structural result, we enumerate $\mathcal{L}$ (and subclasses of it), with respect to the minimum number of crossings or edges in a projection of $L\in \mathcal{L}$. Further, we obtain counting formulas and asymptotic estimates for the connected $K_4$-minor-free link-diagrams, minimal $K_4$-minor-free link-diagrams, and $K_4$-minor-free diagrams of the unknot.
\end{abstract}

\noindent{\bf Keywords:} series-parallel graphs, links,
knots, generating functions, asymptotic enumeration, map enumeration.

\section{Introduction}

The exhaustive generation of knots and links according to their crossing number is a well-established problem in low dimensional geometry. For an account, see~\cite[Chapter 5]{menasco2005handbook}. In the last decades, there has also been interest in properties of random knots and links and their models, as well as random generation of them; see for instance~\cite{chapman2017asymptotic,buck1994random,diao1994random,even2016invariants} or \cite[Chapter 25]{calvo2005physical}. In parallel, various combinatorial and algorithmic questions of more deterministic nature have been addressed, for example in \cite{adams2015knot,chang2017untangling,medina2017number}.

However, it appears that there are very few enumerative results of knots and links in the combinatorics literature. In fact, they are relatively recent and related to the enumeration of prime alternating links, such as~\cite{sundberg1998rate} and~\cite{kunz2001asymptotic}.
Moreover, it seems that there are no known results on interconnections between  graph-theoretic classes and link classes.

The present paper contributes in this direction. We present both enumerative and structural results, the latter relating in a precise way a fundamental class of links, torus links, with the family of \emph{series-parallel} graphs\footnote{We adopt the following definition for series parallel graphs: A graph is \emph{series-parallel} if it can be obtained
from a double edge after a series of edge subdivisions or duplications.} and, more generally, graphs that exclude $K_4$ as a minor ($K_4$-minor-free graphs). The latter is an extensively studied graph class. For instance, it is known that $K_4$-minor-free graphs are exactly the graphs with treewidth at most $2$, while a graph is  $K_{4}$-minor free if and only if all its non-trivial 2-connected 
 components are series-parallel graphs~\cite{bodlaender1998partial},\cite{brandstadt1999graph}. Enumerative results for series-parallel graphs are  available in~\cite{bodirsky2007enumeration}.

Before stating the results, let us give a few preliminary definitions. A \emph{knot} is a smooth embedding of the 1-dimensional sphere $\mathbb{S}^1$ in $\mathbb{R}^3.$
A \emph{link} is a finite disjoint union of knots. A standard way to associate links to graphs is to represent them via {\sl  link-diagrams} that are their
projections to the plane. That way, link-diagrams are  seen as  4-regular maps,
where each vertex corresponds to a crossing of the link with itself and where we mark  the pair of opposite edges that is overcrossing the other. Notice that link-diagrams may contain vertex-less edges.
Clearly, any link has arbitrarily many different link-diagrams. A {\em minimal} link-diagram is one with the minimum possible number of vertices, for the link $L$ it represents. This number is called {\em crossing number}
of $L$.

Our first result is a complete structural characterisation of  {$K_{4}$-minor-free} links, i.e., links that admit some {$K_{4}$-minor-free} link-diagram, via a decomposition theorem (Theorem~\ref{acceptance}) derived after a series of graph-theoretic lemmata. Using this decomposition and analytic techniques of generating functions, we are able to deal with a series of  enumeration problems.

Denote by ${\cal L}$ the set of all $K_{4}$-minor-free link-types. Among them, we distinguish the
subset $\overline{\cal L}$ of the non-split links (i.e., those without disconnected link-diagrams), the subset $\hat{\cal L}$ of links without trivial disjoint components (i.e., those without link-diagrams with vertex-less edges), and the subset ${\cal K}$ of  knots in ${\cal L}$.
For each object in a set of links, we denote by $n$ (resp. $m$) the number of vertices (resp. edges) of a minimal diagram and we define the combinatorial classes $({\cal L},m)$, $(\overline{\cal L},n)$, $(\hat{\cal L},n)$, and $({\cal K},n)$. $\mathcal{L}$ is considered with respect to the number of edges in a minimal diagram, in order to account for the number of trivial components that project to vertex-less edges.

Our enumerative results on link-types are the following. First, for both classes $(\overline{\cal L},n)$ and $(\hat{\cal L},n)$, it holds that the number of elements in the class with $n$ vertices
has asymptotic growth of the form
\begin{align*}
C \cdot n^{-5/2}\cdot \rho^{-n}.\label{unalloyed}
\end{align*}
In both cases $\rho\approx 0.44074$ and the constant $C$ is approximately equal to $ 9.92890$ and $24.96355$  for $\overline{\cal L}$ and  $\hat{\cal L}$ respectively (see Theorem~\ref{transform}).

For the class $({\cal L},m)$, we have to distinguish between even and odd $m$. In both cases the exponential growth is the same and equal to $\sqrt{\rho}$, but the constant $C'$ changes. For even $m$, $C'\approx 251.44816$, and for odd $m$, $C'\approx 166.93220$ (see Corollary~\ref{doublelife}).
%
Finally, the class $({\cal K},n)$ has asymptotic growth of the form
\begin{equation*}C''\cdot n^{-7/4}\cdot \exp (\beta n^{1/2}),\label{prospects}\end{equation*}
where $C''\approx 0.26275$ and $\beta \approx 2.56509$ (see Theorem~\ref{elsewhere}).
The latter follows by manipulation of asymptotic estimates of different classes of integer patitions.

Our next set of results concerns the enumeration of
link-diagrams. Let ${\cal M}$ be the set of all connected $K_{4}$-minor-free link-diagrams, ${\cal M}_{1}$ be the set all minimal connected $K_{4}$-minor-free link-diagrams, and let ${\cal M}_{2}$ be the set of all connected $K_{4}$-minor-free link-diagrams of the unknot. For a link-diagram $R$, we denote by $m(R)$ its number of edges.
We define then the combinatorial classes $({\cal M},m)$, $({\cal M}_1,m)$, and $({\cal M}_2,m)$. We obtain that all these three combinatorial classes follow an asymptotic growth of the form
\begin{align*}
\frac{1}{2m}\cdot C_{l}\cdot m^{-3/2}\cdot\rho^{-m},\label{pursuance}
\end{align*}
%
%
where the constants can be found in Table~\ref{results2}.
\begin{center}\label{results2}
\begin{tabular}{c|c|c}
  Family of diagrams & $\rho$  & $C_l$    \\
     \hline\\[-1em]
       All & 0.31184 & 0.85906  \\
  Minimal & 0.41456 &  0.45938  \\
  Unknot & 0.23626 &  0.95896\\
\end{tabular}
\end{center}

Our strategy to obtain these results has as starting point the equations given in \cite{noy2019enumeration} for 4-regular graphs in the rooted map context. We refine significantly these equations in order to tackle the main difficulty here, which is the crossing structure of our link-diagrams. We first obtain defining systems of equations for the rooted analogues of the aforementioned classes and analyse the corresponding asymptotic behaviour. Then we use techniques from~\cite{richmond1995almost} and~\cite{bender1992submaps}, in order to transfer these results to the unrooted map classes under study.

\paragraph{Structure of the paper.} In Section \ref{universal} we introduce all topological notions and definitions in knot theory that we will use in the rest of the paper. Similarly, in Section \ref{underrate} we state the preliminaries needed for combinatorial enumeration, and in Section~\ref{analysis} we resume most of the analytic tools needed to provide asymptotic estimates. In Section~\ref{monuments}
we prove our structural result for $K_4$-free links and in Section~\ref{enumsection} their enumeration, both exact (by means of generating functions) and asymptotic. In Section \ref{according} we provide enumerative formulas for different kinds of rooted link-diagrams, using tools from map enumeration.
In Section \ref{unrooting} we transfer these results in the unrooted setting. Finally, in Section 7 a list of open problems is discussed.

\paragraph{Note.} All the computations in this paper were performed in \texttt{Maple}. The \texttt{Maple} session can be found in \texttt{https://mat-web.upc.edu/people/juan.jose.rue/research.html}.

\color{black}

\section{Preliminaries}

\subsection{Graph-theoretic Preliminaries}\label{underrate}

All graphs in this paper are multigraphs, i.e., they may have loops or multiple edges.
Given a graph $G=(V,E)$ and a vertex $v\in V$, we denote by $N_G(v) \subseteq V$ the set of neighbors of $v$.
For a vertex subset $A\subseteq V$ we denote by $G-A$ the graph obtained from $G$ by removing the vertices in $A$ and all edges incident with vertices in $A$. Similarly, for a set $B \subseteq E$, we denote by $G-B$ the graph obtained from $G$ by removing the edges in $B$.

A graph $G$ is \emph{$k$-vertex connected} (or shortly, \emph{$k$-connected}) if it has more than $k$ vertices and, if $A$ is a subset of $V$ of size strictly smaller than $k$, then $G-A$ is always connected.
Similarly, a graph $G$ is \emph{$k$-edge connected} if it has more than $k$ edges and, if $B$ is a subset of $E$ of size strictly smaller than $k$, then $G-B$ is always connected.

We say that a graph $H$ is a {\em subdivision} of a graph $G$ if $G$ can be obtained from $G$ by replacing some of its edges with paths having the same endpoints.
Given two graphs $H$ and $G$, we say that $H$ is a \emph{topological minor} of $G$ if it contains as a subgraph some subdivision of $H$.  If $G$ does not contain $H$ as a topological minor, then we say that $G$ is {\em $H$-topological minor free}. $H$ is a {\em minor} of $G$
if $H$ can be obtained from some subgraph of $G$
after contracting edges. It is easy to see that $K_{4}$
is contained as a minor if and only if it is contained as a topological minor. Therefore, $K_{4}$-topological minor free graphs are exactly the $K_{4}$-minor free graphs.

A graph is \emph{outerplanar} if it can be embedded in the plane in such a way that all vertices lie on the outer face. Equivalently, it does not contain a subdivision of $K_4$ or $K_{2,3}$. (see~\cite{harary6graph}).

For every $n\geq 3$, we denote by $\hat{C}_{n}$ the multigraph
obtained if in a cycle of $n$ vertices we replace all edges by double edges. We extend this definition so that $\hat{C}_{2}$ is the graph consisting of two vertices connected with an edge of multiplicity 4, $\hat{C}_{1}$ is a vertex with a double loop, and by convention we say that $\hat{C}_{0}$ is the vertex-less edge (that is the edge without endpoints).

\subsection{Preliminaries for knots and links}\label{universal}

A \emph{knot} $K$ is a smooth embedding of the 1-dimensional sphere $\mathbb{S}^1$ in $\mathbb{R}^3.$
A \emph{link} is a finite union of knots that are pairwise non-intersecting $L=K_1\cup \cdots \cup K_\mu$.
In this situation, each knot $K_i$ is called a \emph{component} of the link $L$. 
Note that there are alternative formulations in the literature \cite[Ch. 1]{cromwell2004knots}, either using polygonal knots or the notion of local flatness, which are equivalent to the one we use.

Two links $L_1$ and $L_2$ are said to be \emph{ambient isotopic} (or equivalent) if there is a continuous map
$h:\mathbb{R}^3\times[0,1]\rightarrow \mathbb{R}^3$, such that, for all $t\in [0,1]$, $h(x,t)$ is a homeomorphism and $h(L_1,0)=L_1$, $h(L_1,1)=L_2$. We then say that $L_1$ and $L_2$ have the same \emph{type} and write $L_1 = L_2$. Note that ambient isotopies preserve the orientation of $\mathbb{R}^3$.

A link equivalent to a set of disjoint circles in the plane is called a \textit{trivial link}. Likewise, a knot equivalent to a circle is called the \textit{trivial knot} or the \textit{unknot}. Two components $C_1,C_2$ of a link $L$ will be called {\em equivalent} if there is an ambient isotopy that maps $L$ to itself and $C_1$ to $C_2$. The latter is an equivalence relation on the components of the link.

\paragraph{Decomposition of links.}

Given two links $L_1,L_2$, their \emph{disjoint sum} is obtained by embedding $L_1$ in the interior of a standard sphere and $L_2$ in the exterior. We denote the resulting link by $L_1 \cupdot L_2$ and call each $L_1,L_2$ a \emph{disjoint component} of $L$. A link -- and, accordingly, all members of its equivalence class -- is \emph{split} if it is the disjoint sum of two links.
\begin{figure}[h!]
\centering \includegraphics[scale=.7]{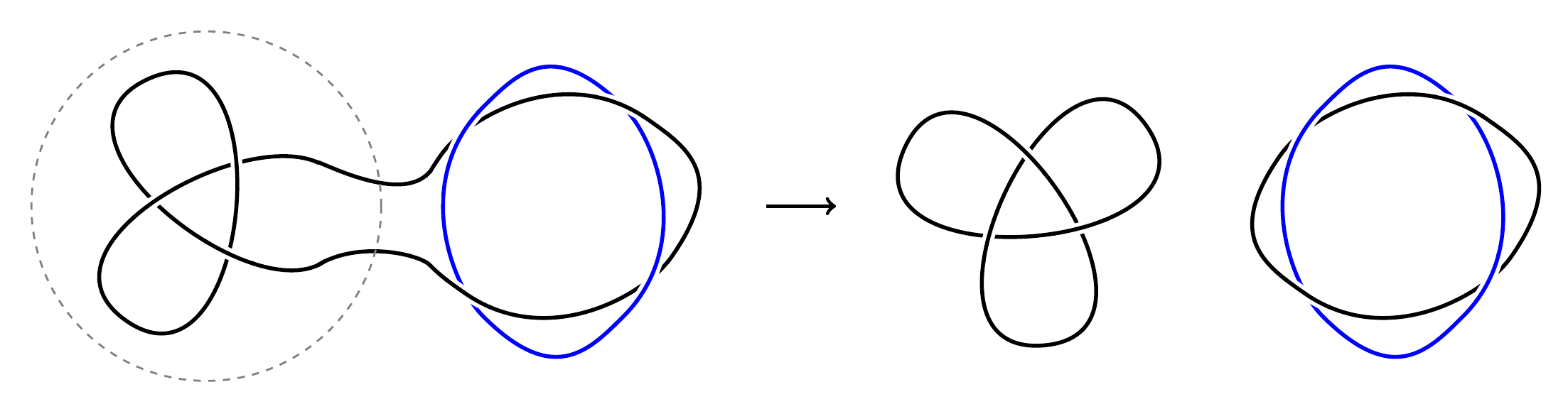}
\caption{A connected sum, corresponding to the links $T(3)$ and $T(-4)$.}\label{frivolity}
\end{figure}
Consider a link $L$ and the sphere $\mathbb{S}^2$ embedded in such a way that it meets the link transversely in exactly two points $P_1$ and $P_2$.
Then we discern two different links $L_1,L_2$ when connecting $P_1$ and $P_2$. The first corresponds to the part of $L$ in the interior of the sphere and the second to the part in the exterior. We then say that $L$ is a \textit{connected sum} with \textit{factors} $L_1,L_2,$ denoted $L_1\# L_2$ (see Figure~\ref{frivolity} for an example). A factor of a link is a \emph{proper factor} if it is not the trivial knot and is not equivalent to the link itself. A link with proper factors is called \emph{composite}. Otherwise, it is called \emph{locally trivial}. Finally, a link is called \emph{prime} if it is non-trivial, non-split, and locally trivial.

  The following two theorems are well known in Knot Theory:
\begin{theorem}[{\cite[Theorem 3.2.1]{kawauchi2012survey}}]
Let $L$ be a link such that $L=L_1\#L_2$ for two links $L_1$ and $L_2$. Then $L$ is trivial if and only if both links $L_1$ and $L_2$ are trivial.
\end{theorem}
\begin{theorem}[{\cite[Theorem 3.2.6]{kawauchi2012survey}}]
A non-split link can be decomposed into finitely many prime links with respect to the connected sum. Moreover, the decomposition is unique in the following sense: If $L_1\# L_2\# \cdots \# L_m = L^{'}_1\# L^{'}_2\# \cdots \# L^{'}_n$ for prime links $L_i\,(i = 1,2,\dots , m)$ and $L^{'}_j \,(j = 1,2, \dots , n)$, then we have $m = n$ and for each $i\in[n]$ and some permutation $\sigma$ of $[n]$,
$L_i = L^{'}_{\sigma(i)} $.\label{masochism}
\end{theorem}

Note that the connected sum of two given knots is only well-defined for oriented knots. However, if they are \emph{invertible}, i.e., are equivalent to themselves with opposite orientation, then it is well defined (see the relevant discussion in~\cite[Ch. 4.6]{cromwell2004knots}). In this case, the connected sum between links is also well-defined, if one specifies the equivalence classes of the components that get connected.

\begin{definition}
Let $\mathcal{L}$ a family of links. We denote by ${\bf dcl}(\mathcal{L})$ the set of finite disjoint sums of links in $\mathcal{L}$. By ${\bf ccl}(\mathcal{L})$, we denote the set of finite connected sums of links in $\mathcal{L}$ that are non-split.
\end{definition}

\paragraph{Maps and link-diagrams.}
We use the term {\em maps} for multigraphs that are embedded
in the sphere and we say that they are {\em 4-regular}
when each vertex is incident to 4 half-edges.

Given a map $G$, we denote its vertex set by $V(G)$ and its edge set by $E(G)$.
Let $G$ a 4-regular map and let $v\in V(G)$.
We denote by $\overline{e}$ the set of points of the plane
corresponding to an edge $e\in E(G)$ and we pick a point
$x_{e}\in \overline{e}$. We call the two connected components of $\overline{e}\setminus \{x_{e}\}$ {\em half-edges of $G$}
corresponding to the edge $e$. We also use the notation $\hat{E}(G)$ to denote the set of half-edges of the
embedding of $G$.
 For every $v\in V(G)$ we denote by $\hat{E}_{v}$ the set of half-edges
containing $v$ in their boundary. Notice  that $\hat{E}_{v}$
is cyclically ordered as indicated by the embedding of $G$.
Two half-edges in $\hat{E}_{v}$ are called {\em opposite}
if they are non-consecutive in this cyclic ordering. Clearly, $\hat{E}_{v}$ contains two pairs of opposite half-edges. A \emph{corner} on a map is the region between two consecutive half-edges around a given vertex.

Two maps are considered to be the same if the first is obtained from the second by a homeomorphism of the sphere which preserves its orientation.
For enumerative purposes, we consider \emph{rooted} maps: a {\em rooted map} is a map with a marked corner; the incident vertex is called the \emph{root vertex}, and the edge following the marked corner in clockwise order around the root vertex is called the \emph{root edge}. Finally, the face that contains the marked corner is the \emph{root face} of the map. Equivalently, a rooted map is defined by orienting an edge in the map (the root vertex corresponds to the initial vertex of the edge) and choosing the root face as the one on the left of the rooted edge.

The next definition introduces the type of enriched maps that we will study in this paper:
\begin{definition}
A {\em link-diagram} is a triple $D=(V,E,\sigma)$, where $G=(V,E)$ is a connected 4-regular map and $\sigma: V(G)\to{\hat{E}(G)\choose 2}$, such that for every $v\in V(G)$, $\sigma(v)$ is a set of two opposite half-edges  of the embedding of $G$.
\end{definition}
We say that a link-diagram $(V,E,\sigma)$ is {\em reduced} if the graph $G=(V,E)$
does not contain any cut-vertex.
Notice that each link-diagram $D=(V,E,\sigma)$ corresponds to a link-type which we denote by $L(D)$.
The link-diagram $D$ is obtained from $L(D)$ by projecting it on the sphere (or equivalently, on the plane).
Moreover, it is a standard fact that for each link-type $L$
there is at least one link-diagram $D$ where $L(D)=L$ (see~\cite[Ch. 3]{cromwell2004knots}).
Given a link-type $L$, we denote by ${\cal D}_{L}$
the set containing every diagram $D$ such that $L(D)=L$.
Let $L$ be a link-type and $D$ a diagram of $L$ with the minimum number of vertices, $n$, over all the  link-diagrams in ${\cal D}_{L}$. $D$ is called a \emph{minimal diagram} and $n$ is called the \emph{crossing number} of $L$.

Finally, we can apply certain local moves on link-diagrams, called \textit{Reidemeister moves}, that do not alter the type of the link.
It is known that, given two link-diagrams that correspond to the same knot, one can be obtained from the other by a sequence of Reidemeister moves \cite{Reidemeister1927}. In Figure~\ref{criterion}, there is a depiction of these moves.

\begin{figure}[h!]
\centering \includegraphics[scale=1.7]{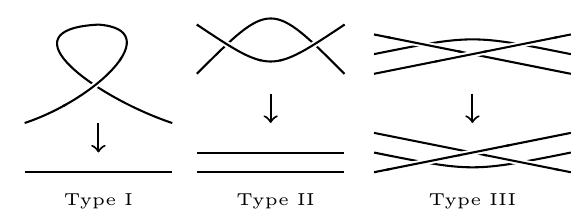}
\caption{The 3 types of Reidemeister moves.}\label{criterion}
\end{figure}

\paragraph{Torus links.}

\emph{Torus links} are links that can be embedded in the standard torus. They are a well-studied class of invertible links. We will give a brief exposition of their properties that we need. For rigorous proofs and general background, the reader is deferred to~\cite[Chapter 7]{murasugi2007knot},~\cite[Chapter 10]{cromwell2004knots},~\cite{kawauchi2012survey}.

Torus links can be classified into types $T(p,q)$, $p,q\in \mathbb{Z}$. We will be interested in types $T(q):=T(2,q)$, which have the following properties. When $q=\pm 1$ or $q=0$, $T(q)$ is the unknot. Otherwise, for $q$ odd it is a prime knot and for $q$ even it is a prime link. In the first case, the crossing number is equal to $q$, while for $q$ even the links have exactly two components. Links of type $T(q)$ can be defined alternatively as links that are embeddable on the torus and cross the meridian cycle two times and the longitude cycle $|q|$ times. The sign of $q$ indicates the two different ways in which the crossings occur, and for all $q>2$, $T(q)\neq T(-q)$. See Figure~\ref{frivolity} for an example of the links $T(3) $ and $T(-4)$.
Note that for every $q$ there is a link-diagram of $T(q)$ with graph $\hat{C}_{|q|}$, which helps to obtain a first visual representation of them.  Moreover, the crossing number of connected sums of torus links is additive~\cite{diao2004additivity}.


\subsection{Analytic tools for combinatorial enumeration}\label{analysis}

Most of the preliminaries in this section can be extensively found in the reference book \cite{flajolet2009analytic}.
\paragraph{Symbolic Method.}
A \emph{combinatorial class} is a pair $(\mathcal{A},|\cdot|)$, where $\mathcal{A}$ is a set of objects and $|\cdot|$ is the \emph{size} of the object. In our setting, the objects will be graphs or maps, and the size will be typically the number of vertices or the number of edges. The latter will also be called the \emph{atoms} of an object. 
We restrict ourselves to the case where the number of elements in $(\mathcal{A},|\cdot|)$ with a prescribed size is finite.
Under this assumption, we define the formal power series $A(z)=\sum_{a\in
\mathcal{A}}z^{|a|}=\sum_{n=0}^\infty a_n z^n$, and conversely, $[z^n]
A(z)= a_n$. $A(z)$ is called the \emph{generating function}
(or shortly the \emph{GF}) associated to the combinatorial class
$(\mathcal{A},|\cdot|)$. 
We will omit the size function whenever it is clear by the context.
Also, we will write $\mathcal{A}_n$ for the set of elements in $\mathcal{A}$ with size $n$, and $|\mathcal{A}_n|=a_n$. The simplest combinatorial class is the \emph{atomic} one, denoted by $\mathcal{Z}$, that contains one object of size one.

The \emph{Symbolic Method} provides a systematic tool to translate combinatorial class constructions
into algebraic relations between their GFs.
The basic constructions are the following.
The \emph{(disjoint) union} $\mathcal{A} \cup \mathcal{B}$ of two
classes $\mathcal{A}$ and $\mathcal{B}$ is defined in the obvious way. The cartesian product $\mathcal{A}
\times \mathcal{B}$ of two classes $\mathcal{A}$ and $\mathcal{B}$ is the set of pairs
$(a,b)$ where $a\in \mathcal{A}$ and $b\in \mathcal{B}$. The size of $(a,b)$ is additive, i.e. it is the
sum of the sizes of $a$ and $b$.
We can define the sequence and the multiset construction of a set $\mathcal{A}$, defined as the set of sequences (resp. multisets) of elements in $\mathcal{A}$, with size again defined additively. Finally, the composition of combinatorial classes $\mathcal{A}\circ \mathcal{B}$ corresponds to the substitution of objects of $\mathcal{B}$ into atoms of objects in $\mathcal{A}$. The size is additive in the objects of $\mathcal{B}$ that are used. The Pointed class $\mathcal{A}^\bullet$ corresponds to  $\mathcal{A}$ where one atom from each object is distinguished (we also call it \emph{rooted} class).
In Table \ref{exchanged} we include all the encodings into generating functions. Note that, in order for the GF encoding of the composition to work, one needs to assume that the atoms are distinguishable. 
\begin{table}[htb]
\begin{center}
\begin{tabular}{c c|c}
  Construction &  & Generating function \\\hline
  Union & $\mathcal{A}\cup\mathcal{B}$ & $A(z)+B(z)$ \\
  Product & $\mathcal{A}\times\mathcal{B}$ & $A(z)\cdot B(z)$ \\
  Sequence & $\mathrm{Seq}(\mathcal{A})$ & $\left(1-A(z)\right)^{-1}$ \\
  Multiset & $\mathrm{Mset}(\mathcal{A})$& $\prod_{a\in \mathcal{A}}(1-z^{|a|})^{-1}=\exp\left(\sum_{r=1}^{\infty}\frac{1}{r}A(z^r)\right)$\\
Composition& $\mathcal{A}\circ \mathcal{B}$ & $A(B(z))$\\
Pointing& $\mathcal{A}^\bullet$ & $\partial _zA(z)$
\end{tabular}
\caption{Table of combinatorial constructions, and their generating function counterparts.}\label{exchanged}
\end{center}
\end{table}	
\paragraph{Complex analysis and generating functions.}
We apply singularity analysis over generating functions to obtain asymptotic estimates.
The main reference here is again \cite{flajolet2009analytic}.
We say that a domain in $\mathbb{C}$ is \emph{dented} at a value $\rho>0$ if it is a set of the form $\Delta(\theta,R,\rho)=\{z\in \mathbb{C}: |z|<R,\,\arg(z-\rho)\notin [-\theta,\theta]\}$ for some real number $R>\rho$ and some positive angle $0<\theta<\pi/2$. We call \emph{dented disk} at $\rho$ a set of the form $\Delta(\theta,R,\rho)\cap B(\rho, r)$, where $B(\rho, r)$ is a ball of radius $r$ around $\rho$.

Let $f(z)$ be a generating function which is analytic in a dented domain at $z=\rho$.  The singular expansions we encounter in this paper are of the form
\begin{equation*}\label{sing}
f(z) = f_0 +f_1Z+f_2 Z^2 + f_3 Z^3 +f_4 Z^4 + \cdots  + f_{2k}Z^{2k} + f_{2k+1}Z^{2k+1} + O\left(Z^{2k+2}\right),
\end{equation*}
where $Z = \sqrt{1 - z/\rho}$ and $k=0$ or $k=1$.
 That is, $2k + 1$ is the smallest odd integer $i$ such that $f_i \ne 0$. The even powers of $Z$ are
analytic functions and do not contribute to the asymptotic behaviour of $[z^n]f(z)$. The number $\alpha = (2k+1)/2$ is called the
\textit{singular exponent}. If there is no other complex value of the same or smaller modulus on which such an expansion holds, we can apply the Transfer Theorem~\cite[Corollary VI.1]{flajolet2009analytic}
and obtain the estimate
\begin{equation}\label{everybody}
[z^n]f(z) \sim c \cdot n^{-\alpha-1}\cdot \rho^{-n},
\end{equation}
where $c = f_{2k+1}/\Gamma(-\alpha)$ and $\Gamma$ denotes the Gamma function. If there is a finite number of singularities $\rho _i$ of minimum modulus and corresponding expansions in dented disks at $\rho _i$, the same estimates apply and the contributions are added~\cite[Theorem VI.5]{flajolet2009analytic}. In order to obtain such expansions, we will use the following theorem.
\begin{theorem}[{\hspace{-.02cm}\cite[Proposition 1, Lemma 1]{drmota1997systems}}]\label{drmota}
Suppose that $F(z, y)$ is an analytic function in $z,y$ such that $F(0, y)\equiv 0$, $F(z, 0)\not\equiv 0$ and all Taylor coefficients of $F$ around $0$ are real and nonnegative. Then, the unique solution $y=y(z)$ of the functional equation $y=F(z,y)$ with $y(0)=0$ is analytic around $0$ and has nonnegative Taylor coefficients $y_k$ around $0$. Assume that the region of convergence of $F(z,y)$ is large enough such that there exist nonnegative solutions $x=\rho$ and $y=s$ of the system of equations
$\{ y = F(z,y), 1 = \frac{\partial }{\partial y}F(z,y),\}$
where $\frac{\partial }{\partial z}F(\rho,s)\neq0$ and $\frac{\partial }{\partial yy}F(\rho,s)\neq 0$.  Then, $y(z)$ has a representation of the form
\begin{equation} y(x)=
q(z)+h(z)\sqrt{\bigg( 1-\frac{z}{\rho} \bigg)}. \label{singular}\end{equation}
in a dented disk around $z=\rho$.
If also $y_k>0$ for large enough $k$, then $\rho $ is the unique singularity of $f$ on its radius of convergence and there exist functions $q(z), h(z)$ which are analytic around $z=\rho$, such that $y(z)$ is analytically continuable in a dented domain at $\rho $.
\end{theorem}
We note that the coefficients in~\eqref{singular} are explicitly computable. One coefficient is given in~\cite{drmota1997systems}, namely $$h_0 = \sqrt{\frac{2\rho \frac{\partial }{\partial x}F(\rho,s)}{\frac{\partial }{\partial yy}F(\rho,s)}},$$ and the rest are easy to compute similarly (see the associated \texttt{Maple} session).

\section{Structure of $K_{4}$-minor-free link-diagrams}\label{monuments}

\label{peregrina}

We say that a link-type $L$ is {\em $K_{4}$-minor free} if there exists a diagram in ${\cal D}_{L}$ that is $K_{4}$-minor free (recall that ${\cal D}_{L}$ denotes all possible diagrams arising from $L$).
Given an $i\in\Bbb{N}$, we denote by ${\cal D}_{\geq i}$ the set of all link-diagrams whose graph is $\hat{C}_{j}$ for some $j\geq i$.

Let $D_{i}=(V_{i},E_{i},\sigma_{i}), i\in\{1,2\}$ be two diagrams, where $V_i\neq \emptyset $.
We say that a diagram $D(V,E,\sigma )$ is a  {\em 2-edge-sum} of $D_{1}$ and $D_{2}$ if $D$ can be created from $D_{1}$ and $D_{2}$ as follows:
we pick two edges $e_{1}\in E_{1}$ and $e_{2}\in E_{2}$, we remove them, and add two edges $f_{1}$ and $f_{2}$ such that both $f_{1},f_{2}$ have endpoints from both $e_{1}$ and $e_{2}$, and such that the resulting embedding remains plane. The $\sigma $ function is preserved, i.e., for all $v\in V_i$, $\sigma (v)\cap \sigma _i(v)\neq \emptyset$.

Let ${\cal D}$ be a set of link-diagrams.
We define the {\em closure} of ${\cal D}$, denoted by ${\bf cl}({\cal D})$ with respect
to 2-edge sums as the set containing every diagram
$D$ such that
\begin{itemize}
\item either $D\in {\cal D}$ or
\item  there exists $D_{1},D_{2}\in{\cal D}$ such that $D$ is a 2-edge sum of $D_{1}$ and $D_2$.
\end{itemize}
From now on, we denote by ${\cal D}$ the set of all  link-diagrams
whose graph is  $K_{4}$-minor-free.
\begin{lemma}
\label{archetype}
Let $G=(V,E)$  be a $4$-regular, $K_{4}$-minor-free and 3-edge-connected graph. Then $G$ is outerplanar.
\end{lemma}

\begin{proof}
Suppose to the contrary that $G$ is not outerplanar, and hence it contains as a subgraph some subdivision $H$ of $K_{2,3}$.
Let $v_{1}$ and $v_{2}$ be the vertices of $H$ that have degree $3$.
Let also $P_{1}$, $P_{2}$, and $P_{3}$ the paths that are the connected components of $H-\{v_1,v_2\}$.

Let $G^-=G-\{v_1,v_2\}$.
We first observe that none of the connected components of $G^-$ contains more than one of the paths in $\{P_{1},P_{2},P_{3}\}$, as this would imply the existence of a path between vertices of these two paths in the connected component that contains them.
This would imply the existence of a copy of $K_{4}$ as a topological minor, a contradiction.

Using the above claim, we deduce that $G^-$ has at least 3 connected components $C_{1},C_{2},C_{3}$, where $P_{i}$ is a subgraph of $C_{i}, i\in[3]$.
Let $F\subseteq E(G)$ be the set of edges that are incident to either $v_{1}$ or $v_{2}$.
Clearly, by the 4-regularity of $G$ and the existence of $P_i$, $F$ contains 7 or  8 edges, depending on whether $v_1$ and $v_2$ are adjacent or not.
Moreover, because of the 3-edge-connectivity of $G$, for each $i\in[3]$ there are at least 3 edges in $F$ that are incident to vertices in $C_{i}$ and this implies that $|F|\geq 9$, a contradiction.
\end{proof}
\begin{lemma}
\label{situation}
Every $2$-connected, 4-regular, $K_{4}$-minor-free, and $3$-edge-connected graph on $n\geq 1$ vertices is isomorphic to $\hat{C}_{n}$.
\end{lemma}
\begin{proof}

We examine the non-trivial case where $n\geq 3$.
From Lemma~\ref{archetype},  $G$ is $K_{2,3}$-free,  therefore it is outerplanar and  can be embedded in the plane so that all its vertices
lie on its unbounded face $F$.
Let $E_{\rm out}$ be the set of the edges of $G$  that  are
incident to $F$.
For each edge $e$ in $E_{\rm out}$, we denote by $F_{e}$ the face that is incident to $e$ and is different from $F$.

We next claim that for every $e\in E_{\rm out}$, $F_{e}$ is incident to exactly two edges.
Suppose to the contrary that this is not correct for some $e\in E_{\rm out}$ with end-vertices $x$ and $y$.
Let $z$ be a vertex incident to $F_{e}$ that is different from $x$ and $y$.
Notice that $z$ is a cut-vertex in the graph $G^-=G- \{e\}$, that places $x,y$ in different connected components in $G^{-}-\{z\}$.
Let us call them $C_x$ and $C_y$.
Since $z$ has degree 4, for one of them it holds that $|V(C_j)\cap  N_{G^-}(z)| \leq 2$, say for $C_x$.

Let $S=V(C_x)\cap  N_{G^-}(z)$ and observe that $\{\{z,w\}\mid w\in S\}\cup \{e\}$ is an edge separator of $G$ of size $\leq 3$ (an edge separator is a set of edges whose removal increases the number of connected components).
As $G$ is connected and every edge separator of a 4-regular graph contains an even number of edges, we obtain that $|S|=2$, a contradiction to the 3-edge connectivity of $G$.

We just proved that $G$ contains $\hat{C}_{n}$ as a spanning subgraph (i.e., a subgraph with the same set of vertices). The fact that $G$ does not contain more edges
than $\hat{C}_{n}$ follows from the fact that $\hat{C}_{n}$ is already $4$-regular.
\end{proof}
\begin{lemma}
${\bf cl}({\cal D}_{\geq 1})$ is the  set of all reduced and connected $K_{4}$-minor-free link-diagrams. \label{inscribed}
\end{lemma}
\begin{proof}
We set ${\cal C}={\bf cl}({\cal D}_{\geq 1})$.
Suppose that there exists a  $D=(V,E,\sigma)$ that is a reduced $K_{4}$-minor-free link-diagram and does not belong to ${\cal C}$. Let $D$ be such a diagram where $|V|$ is minimized.
If $G=(V,E)$ is 3-edge-connected then, by Lemma~\ref{situation}, $G$ is isomorphic to $\hat{C}_{n}\in{\cal D}_{\geq 1}\subseteq {\cal C}$, a contradiction. Therefore $G$ has an edge-cut consisting of two edges $e_{1}=\{x_{1},x_{2}\}$ and $e_{2}=\{y_{1},y_{2}\}$.
As $D$ is reduced, $G$ has no cut-vertices, therefore $x_{1},x_{2},y_{2},y_{2}$ are pairwise distinct.
Let $G_{1}^-$ and $G_{2}^-$ be the connected components of $G-\{e_{1}, e_{2}\}$
and without loss of generality, we assume that $x_{i},y_{i}\in V(G_{i}^{-}), i\in[2]$.
Let $G_{i}$ be the graph obtained from $G_{i}^-$
after adding the edge $\{x_{i},y_{i}\}$, $i\in[2]$.
We also set $\sigma_{i}=\sigma|_{V(G_{i})}, i\in[2]$.
Observe that $D$ is a 2-edge sum of $D_{1}=(G_{1},\sigma_{1})$ and $D_{2}=(G_{1},\sigma_{2})$.
Moreover both $G_{1}$ and $G_{2}$ are 2-connected, $K_{4}$-minor free, and 4-regular.
As $G_{1}$ and $G_{2}$ have both fewer vertices than $G$, by the minimality of the choice of $D$, we have that $D_{1}, D_{2}\in{\cal C}$, therefore $D\in{\cal C}$, a contradiction.

 Suppose there exists a diagram $D\in \mathcal{C}$ that either is not reduced or contains $K_4$ as a minor. We again choose
 such a $D=(V,E,\sigma)$ where $|V|$ is minimized.
This cannot be of the form of $\hat{C}_{n}$, as all such diagrams are biconnected and $K_{4}$-minor free.
If $D\not\in {\cal D}_{\geq 1}$, then there are $D_{1},D_{2}\in{\cal C}$ with smaller vertex set, such that
$D$ is the 2-edge sum of $D_{1}$ and $D_{2}$. The latter diagrams are reduced and $K_4$-minor-free, because of the minimality of $D$. Consequently, $D$ is also $K_4$-minor-free, since the 2-edge sum operation does not create any new $K_4$ in $D$. Moreover, the 2-edge sum operation maintains 2-connectivity. These two facts contradict the choice of $D$.
\end{proof}

Let ${\cal T}_{2}={\bf ccl}(\bigcup_{q\in\Bbb{Z}}\{T(q)\})$.
Let $\mathcal{L}$ be the class of links that have a $K_4$-minor-free link-diagram, namely, ${\cal L}=\{L\mid {\cal D}_{L}\cap {\cal D}_{}\neq\emptyset\}$ (recall that ${\cal D}$ is the set of all link-diagrams
whose graph is  $K_{4}$-minor-free).
The following theorem gives a structural decomposition of links in $\mathcal{L}$ (recall the definition of $\bf dcl$ in Definition 1).

\begin{theorem}
\label{acceptance}
$\mathcal{L}={\bf dcl}({\cal T}_2)$.
\end{theorem}

\begin{proof}
It is clear that both classes are closed under disjoint sums, so it is enough to prove the theorem for non-split links in $\mathcal{L}$ which, for the purpose of this proof, are  denoted  by $\mathcal{L}'$.

We first prove that $\mathcal{L}'\subseteq {\bf ccl}({\cal T}_2)$. Let $L\in \mathcal{L}'$ and non-split. Then it has a diagram that is $K_4$-minor-free. Let us pick a diagram $D_L$ of minimal $|V|$. This is reduced, so $D_L\in {\bf cl}({\cal D}_{\geq 1})$ by Lemma~\ref{inscribed}. Then $D_L$ is either some $\hat{C}_i$ or a series of consecutive 2-edge sums between $\hat{C}_i$. The 2-edge sum operation can be translated to the connected sum operation in the corresponding links. Thus, either $L$ is a torus link $T(q)$, $q\in \mathbb{Z}\setminus \{0\}$, or the result of a series of connected sums of such torus links, i.e., $L\in {\bf ccl}({\cal T}_2)$.

We now prove that ${\bf ccl}({\cal T}_2)\subseteq \mathcal{L}'$. Let $T\in  {\bf ccl}({\cal T}_2)$ and non-split, i.e. $T=T_1\# ...\# T_n$, where $T_i\in \mathcal{T}_2$ and prime. The claim is shown by induction on $n$. If $n=1$, i.e. $T$ is prime, then the claim is true. Suppose that the claim is true for $n<k$ and let $T=T_1\# ...\# T_{k}$, $T'=T_1\# ...\# T_{k-1}$, and $C$ the component of $T$ on which $T_k$ is connected. Then $T'$ belongs in $\mathcal{L}$ by the induction hypothesis, thus it has a $K_4$-minor-free link-diagram $D$. We know there is an $i$ such that $\hat{C}_i$ is a diagram of $T_k$ with these properties. We embed $\hat{C}_i$ in a face adjacent to an edge of $C$ and perform a 2-edge sum operation. The resulting diagram remains $K_4$-free and represents the link $T$: the way the half-edges were connected does not matter, since the class ${\cal T}_2$ is a class of invertible links.
 \end{proof}




\section{Enumeration of knots and links}\label{enumsection}

Recall that $\mathcal{L}$ is the set of link-types that have a $K_4$-minor-free link-diagram.
Let $\mathcal{K}, \overline{\mathcal{K}}$ be, respectively, the set of knot types in $\mathcal{L}$ and the set of prime knot types in $\mathcal{L}$.
We denote by $\overline{\mathcal{L}}$ the set of non-split link-types in $\mathcal{L}$, and $\hat{\mathcal{L}} $ the set of the link-types in $\mathcal{L}$ with no trivial disjoint components.

\subsection{Enumeration of $\mathcal{L}$}

In this section, we enumerate the combinatorial classes $(\mathcal{L}, m)$, $(\overline{\mathcal{L}}, n)$, $(\hat{\mathcal{L}}, n)$, $(\overline{\mathcal{K}}, n)$, $(\mathcal{K}, n)$, where $m$ is the number of edges in a minimal diagram of a link and $n$ is the crossing number. We denote by $L(z)$, $\overline{L}(z)$, $\hat{L}(z)$, $\overline{K}(z)$, and $K(z)$ the corresponding generating functions (according to the size considered on each class).
Notice that it is not possible to enumerate $\mathcal{L}$ with respect to crossing number: the number of links with a given crossing number is infinite, since the disjoint sum of any such link and a trivial link of arbitrarily many components has the same crossing number.

Let $\mathcal{G}$ be the combinatorial class of unrooted, unlabelled trees, with size equal to the number of vertices. Consider the sets $A= \{2\nu +1: \nu\in \mathbb{Z}\}\setminus \{1,-1\}$, $B=\{ 2\nu : \nu\in\mathbb{Z}\}\setminus \{0,-2\}$. For $T\in \mathcal{G}$, consider all possible labelings of $T$, such that the vertices are labeled with a multiset of $A$ or $1$, and each edge of $T$ is labeled with a number in $B$. We consider two such labelled trees equivalent if there is a graph isomorphism that identifies them as graphs and also identifies their labels. Let $\mathcal{T}$ be the set of the resulting equivalence classes. We define the size of a label $i=i_1,...,i_k$ to be the sum of the absolute values $|i_j|$, and the size of a tree in $\mathcal{T}$ to be the sum of all labels, apart from the 1-label. These labels will be used to encode crossing numbers. See Figure \ref{lakehurst} for an example of an object in $\mathcal{T}$, of size 67.

\begin{figure}[h!]\centering
  \includegraphics[scale=1]{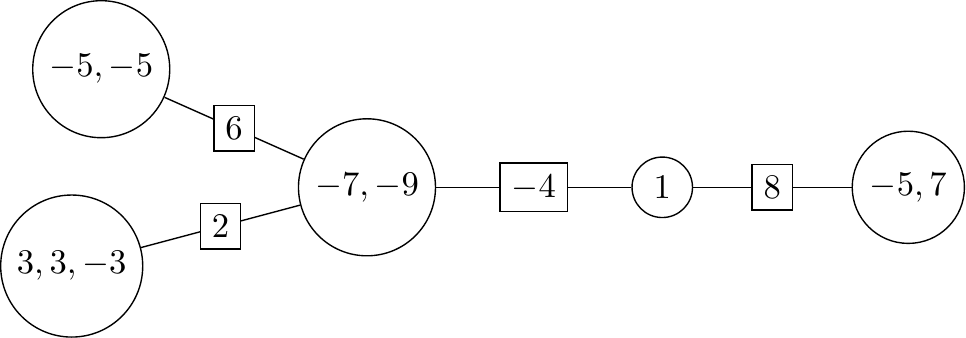}
  \caption{An element of $\mathcal{T}$. The edge labels are drawn inside a square.}\label{lakehurst}
\end{figure}

%
\begin{proposition}\label{bannkreis}
\normalsize There is a size-preserving bijection between $\overline{\mathcal{L}}$ and $\mathcal{T}$.
\end{proposition}
\begin{proof}
%
Let $L\in \overline{\mathcal{L}}$. By Theorem~\ref{acceptance}, there exist prime torus links $T_j=T(q_j)$ such that
\begin{equation*}\label{intensify} L= T_1\# ...\# T_r\end{equation*}
and $L_j=T_1\# ...\# T_j$ is non-split for any $j$.
Let $C_1,C_2....$ be the components of the links $T_1, ..., T_r$, with some arbitrary numbering. Notice that for every $i\in[r]$, $T_{i}$ contains at most two components $C_j$, so there are at most $2r$ of them. For every such component, we write $L(C_j)=T_i$. Each time a connected sum is realised between $L_k$ and $T_{k+1}$, one component of $T_{k+1}$ is identified with a component of $L_k$. Consider the corresponding equivalence relation, i.e., two components $C_i,C_j$ are in the same equivalence class if they are identified in $L$. Let $I_{1},...,I_m$ be the equivalence classes. We define $P(I_j)$ as the multiset of prime torus knots that belong to $I_j$, formally,
 \[P(I_j)\!:=\{(i,q)|\text{~for } i \text{ components }C\in I_j ,\text{ it holds\! } L(C)=T(q), |q|\geq 3,\text{ odd}\}.\]

Let $G_L(V,E)$ be a graph where $V=\{I_1,...,I_m\}$ and there is an edge $I_iI_j$ if and only if there is a link $T_l= T(q_l)$ such that one of its components belongs in $I_i$ and the other belongs in $I_j$. Notice that such a link is unique when it exists, so we can refer to $q_l$ as $q_{ij}$. Moreover, $q_{ij}$ is even. Let $T_L$ be the graph $G_L$, where the vertices $I_i$ have the label $P(I_i)$ and the edges $I_iI_j$ have the label $q_{ij}$ (If $P(I_i)=\emptyset $, the label is 1). Then, $T_L\in \mathcal{T}$.

Let $\phi :\overline{\mathcal{L}}\rightarrow\mathcal{T}$ such that $\phi (L)=T_L$. We first show that $\phi $ is well defined. Suppose that $L_1= T_1\# ...\# T_r=T'_1\# ...\# T'_r=L_2$. Let $G_i^1,G_j^2$ be the components of $L_1$ and $L_2$, corresponding to the equivalence classes $I^1_i,I^2_j$. Since $L_1=L_2$, there is a permutation $\sigma$ of $[n]$, such that there is an ambient isotopy of $\mathbb{R}^3$ that identifies $G_i^1$ with $G^2 _{\sigma (i)}$ for all $i$. Then, the labels on the vertices $I_i^1,I _{\sigma (i)}^2$ are the same, because of the uniqueness of factorisation  in knots (Theorem~\ref{masochism}). Moreover, an edge $I^1_{i}I^1_{j}$ exists if and only if $I^2_{\sigma (i)}I^2_{\sigma (j)}$ exists, and the label on it is the same; otherwise, it holds that
$$L_1\setminus \left\{\bigcup_{h\in[m]\setminus\{i,j\}} G_h^1\right\}\neq L_2\setminus \left\{\bigcup_{h\in[m]\setminus\{\sigma(i),\sigma(j)\}} G_h^2\right\}$$
for some $i,j$. This is a contradiction to the fact that $L_1=L_2$.

We will show that $\phi $ is a bijection. Given $T\in \mathcal{T}$, construct the following link $L_T$: pick $v\in T$ and consider a trivial knot $K_v$. Observe that $v$ has a label that is a multiset of odd numbers $n_i$. Perform all the connected sums between $K_v$ and $T(n_i)$. Then see the labels on the edges $va_i$ neighbouring $v$ and their labels $m_i$, and perform all the connected sums between $K_v$ and $T(m_i)$ ($K_v$ is no longer trivial, unless the label of $v$ was $1$). For each connected sum with $T(m_i)$, consider the component that is not identified with $K_v$. This is a new component that, when seen on its own, is a trivial knot $K_i$. We can then apply the same process beginning from the vertex $a_i$ and the trivial knot $K_i$, looking now at the tree that contains $a_i$ in the forest $T-v$. Continue recursively. By construction $\phi (L_T )=T$, hence $\phi$ is surjective.

Notice that if $\phi (L)=T_L$, then $T_L$ corresponds to a connected sum decomposition of $L$, which defines a link uniquely by Theorem~\ref{masochism}. Consequently $\phi $ is injective.

To see that $\phi$ also preserves size, recall that the crossing number of connected sums of torus links is additive.
\end{proof}

\begin{figure}
\centering
\includegraphics[scale=.55]{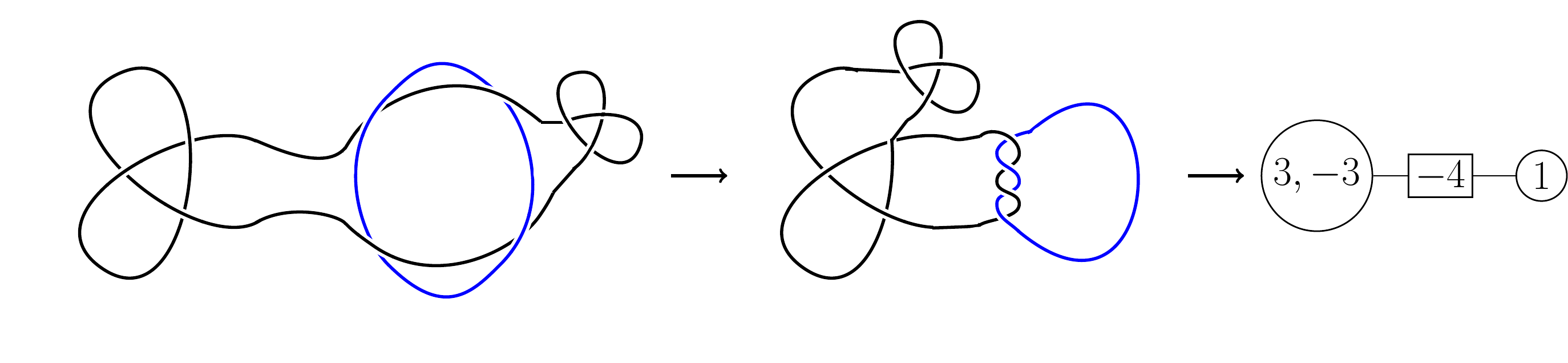}
\caption{An example of the bijection described in Lemma~\ref{bannkreis}. Originally we see a connected sum $T(3)\sharp T(-3)\sharp T(-4)$. We then group $T(3)$ and $T(-3)$ together, since they belong to the same connected component. One can then flip to the right the blue circle and obtain the second picture, which entails the final tree.}
\end{figure}

The latter proposition implies that counting objects in $\overline{\mathcal{L}}$ is equivalent to counting objects in $\mathcal{T}$. We now aim to obtain functional equations that define uniquely the generating functions under study. Let us start with $\overline{K}(z)$, the generating function associated to $\overline{\mathcal{K}}$ (prime torus knots $T(2i+1),$ $i\in \mathbb{Z}\backslash\{0,-1\}$), where $z$ marks crossings. Immediately,
$$\overline{K}(z)= 2\cdot\sum _{i\geq 1}z^{2i+1}=\frac{2z^3}{1-z^2}.$$
Moreover, every object in $\mathcal{K}$ is defined uniquely by a multiset of prime torus knots, therefore $\mathcal{K}=\mathrm{Mset}(\overline{\mathcal{K}})$ and
\begin{equation}\label{mechanics}
K(z)=\exp \Bigg( \sum _{k\geq 1}\frac{1}{k}\overline{K}(z^k)\Bigg)=\exp \Bigg( \sum _{k\geq 1}\frac{1}{k}\frac{2z^{3k}}{1-z^{2k}}\Bigg).
\end{equation}
The first terms of $K(z)$ are the following:
\[K(z)=1+2\,{z}^{3}+2\,{z}^{5}+3\,{z}^{6}+2\,{z}^{7}+4\,{z}^{8}+6\,{z}^{9}+7
\,{z}^{10}+8\,{z}^{11}+13\,{z}^{12}+14\,{z}^{13}+19\,{z}^{14}+26\,{z}^
{15}+\cdots\]

To give a visual representation of what these knots look like, recall that any torus knot $T(q)$ can be represented on the plane by the graph $\hat{C}_{|q|}$. For instance, you can see $T(3)$ in Figure~\ref{frivolity} (known also as \emph{trefoil knot}). If we flip all its crossings we obtain $T(-3)$. This explains the term $2z^3$. Note that all the even terms are composite knots, for instance $3z^6$ comes from $T(3)\sharp T(3), T(-3)\sharp T(-3)$, and $T(3)\sharp T(-3)$.

Let $\mathcal{E}$ denote the combinatorial class of all possible edge labels. Then its generating function is
$$E(z) = z^2+ \frac{2z^4}{1-z^2}.$$
We will now obtain the generating function associated to $\overline{\mathcal{L}}$, denoted by $\overline{L}(z)$. To that end, we will use two standard combinatorial theorems, namely P\'olya's Enumeration Theorem and the Dissymmetry Theorem for trees. Both can be found in~\cite{bergeron1998combinatorial}.

\begin{proposition}
Let $\mathcal{F}=\mathcal{G}^{\bullet }\circ (\mathcal{E}\times \mathcal{K}) $, where $\mathcal{G}$ is the class of unrooted, unlabelled trees (counted according to vertices), and denote by $F(z)$ the generating function associated to $\mathcal{F}$. Then,
\begin{equation}\label{alertness}
\overline{L}(z)=\frac{F(z)}{E(z)}+\frac{E(z)}{2}\bigg( -\frac{F(z)^2}{E(z)^2}+\frac{F(z^2)}{E(z^2)}\bigg).
\end{equation}
\end{proposition}
\begin{proof}
We would like to associate an object of $\mathcal{K}$ and an object of $\mathcal{E}$ to each vertex of a tree $T\in\mathcal{G}^\bullet$. However, we cannot encode this with generating functions by a direct composition $G(K\cdot E)$. The reason is that the vertices of $T$ may not all be distinguishable. Therefore, we will use cycle index sums\footnote{The cycle index series of a combinatorial class $\mathcal{F}$ is the formal power series (in an infinite number of variables) $Z_F(x_1,x_2,x_3,...)=\sum _{n\geq 0}\frac{1}{n!}\big( \sum _{\sigma \in \mathcal{S}_n}\mathrm{fix} F[\sigma ]x_1^{\sigma _1}x_2^{\sigma _2}x_3^{\sigma _3}\cdots \big) $, where $\mathcal{S}_n$ denotes the group of permutations of $[n]$, $\sigma _i$ is the number of cycles of length $i$ in $\sigma$, and $\mathrm{fix}F[\sigma ]$ is the number of objects in $\mathcal{F}$ for which $\sigma $ is an automorphism. For more details on this setting, see~\cite[Chapter 1]{bergeron1998combinatorial}.}. The cycle index sum of $\mathcal{G}^{\bullet}$ is known to satisfy the following functional equation in infinitely many variables (see~\cite[Chapter 4.1]{bergeron1998combinatorial}):
\begin{equation}\label{inventory}
\mathcal{Z}_{\mathcal{G}^{\bullet}}(s_1,s_2,...)=s_1\exp \bigg( \sum _{k\geq 1}\frac{1}{k}\mathcal{Z}_{\mathcal{G}^{\bullet}}(s_k,s_{2k},...) \bigg) .
\end{equation}
We can now obtain the ordinary generating function of $\mathcal{F}=\mathcal{G}^{\bullet }\circ (\mathcal{E}\times \mathcal{K}) $. By P\'olya's Enumeration Theorem (see also~\cite[Section 1.4, Theorem 2]{bergeron1998combinatorial}), the latter satisfies the equation
$F(z)=\mathcal{Z}_{\mathcal{G}^{\bullet}}(f(z),f(z^2),...)$, where $f(z)=E(z)K(z)$.

A tree $T\in \mathcal{F}$ is equivalent to a tree in $\mathcal{T}^\bullet$ (pointing on a vertex) where all labels are on the vertices; a label on an edge $e$ is on the vertex in $e$ that is furthest from the root, and the root-vertex has an extra edge label. We eliminate the extra label from the enumeration, dividing $F(z)$ by $E(z)$. We obtain $T^\bullet (z)=\frac{F(z)}{E(z)}$.

We can obtain an expression for $T(z)$ by $T^\bullet (z)$, using the Dissymmetry Theorem for Trees.
Given a family of trees $\mathscr{T}$ denote by $\mathscr{T}^{\bullet}$, $\mathscr{T}^{\bullet-\bullet}$, and $\mathscr{T}^{\bullet\rightarrow\bullet}$ be the same family with a rooted vertex, a rooted edge and a rooted and oriented edge. Let $T(z)$, $T^{\bullet}(z)$, $T^{\bullet-\bullet}(z)$, and $T^{\bullet\rightarrow \bullet} (z)$ the corresponding generating functions. Then, the Dissymmetry Theorem for trees states that
\begin{equation}T(z)=T^{\bullet}(z)+T^{\bullet-\bullet} (z)-T^{\bullet\rightarrow \bullet} (z).\label{treedis}\end{equation}
In $\mathcal{T}$, it holds that $T^{\bullet}(z)=\frac{F(z)}{E(z)}$, $T^{\bullet\rightarrow \bullet} (z)=\frac{E(z)F(z)^2}{E(z)^2} $, and \[T^{\bullet-\bullet} (z)=\frac{E(z)}{2}\bigg( \frac{F(z)^2}{E(z)^2}-\frac{F(z^2)}{E(z^2)}\bigg) +\frac{E(z)F(z^2)}{E(z^2)},\] where the common factor $E(z)$ encodes the label of the rooted edge. Substituting these expressions in Equation~\eqref{treedis} and using Proposition~\ref{bannkreis}, we obtain the indicated relation for $T(z)$ and then for $\overline{L}(z)$.
\end{proof}

The first terms of $\overline{L}(z)$ are the following:
$$\overline{L}(z)=1+{z}^{2}+2\,{z}^{3}+3\,{z}^{4}+4\,{z}^{5}+9\,{z}^{6}+12\,{z}^{7}+26
\,{z}^{8}+40\,{z}^{9}+82\,{z}^{10}+136\,{z}^{11}+280\,{z}^{12}+\dots
$$
\begin{lemma}\label{lstar}
It holds that $\hat{L}(z)= \exp \left(\overline{L}(z)-1\right) \exp \left( \sum _{k\geq 2}\frac{1}{k}\left( \overline{L}\left(z^k\right)-1\right)\right)$.
\end{lemma}
\begin{proof}
Immediate, since links in $\hat{\mathcal{L}}$ are multisets of links in $\overline{\mathcal{L}}$, excluding the trivial knot.
\end{proof}

The first terms of $\hat{L}(z)$ are the following:
\[\hat{L}(z)=1+{z}^{2}+2\,{z}^{3}+4\,{z}^{4}+6\,{z}^{5}+16\,{z}^{6}+24\,{z}^{7}+56
\,{z}^{8}+98\,{z}^{9}+208\,{z}^{10}+382\,{z}^{11}+805\,{z}^{12}+\dots
\]
We would like to study $K_4$-free link types by the number of edges of a minimal diagram, so as to account also for trivial components. We obtain the following lemma.
\begin{lemma}\label{finalfunc}
For the combinatorial class $\mathcal{L}$ with size equal to the number of edges in a minimal diagram, it holds that
\[L(z)=\frac{\hat{L}(z^2)}{1-z}-1.\]
\end{lemma}
\begin{proof}
Immediate, since a link-diagram  of $n$ vertices without trivial components has $2n$ edges and there is one choice for the number of trivial components that are added. We also remove $1$ as there are not objects in $\mathcal{L}$ with 0 edges. \end{proof}

The first terms of $L(z)$ are the following:
\[L(z)=z+{z}^{2}+{z}^{3}+2\,{z}^{4}+2\,{z}^{5}+4\,{z}^{6}+4\,{z}^{7}+8\,{z
}^{8}+8\,{z}^{9}+14\,{z}^{10}+14\,{z}^{11}+30\,{z}^{12}+30\,{z}^{13}+\cdots
\]

\subsection{Asymptotic analysis}
We proceed now to get asymptotic estimates from the previous generating functions.

\begin{theorem}\label{transform}
The following asymptotic estimates hold:
\begin{align*}
[z^n]\overline{L}(z)\sim \frac{c_1}{\Gamma (-3/2)}\cdot n^{-5/2}\cdot\rho ^{-n},\,\,\mbox{$[z^n]$}\hat{L}(z)\sim \frac{c_2}{\Gamma (-3/2)}\cdot n^{-5/2}\cdot\rho ^{-n},
 \end{align*}
where $c_1,c_2$ are constants and $\Gamma$ is the Gamma function; in particular, $ \rho\approx 0.44074$ ($\rho^{-1}\approx 2.26891$), $c_1\approx 23.46469$, $c_2\approx 58.99565$.
\end{theorem}
\begin{proof}
Let $f(z)=E(z)K(z)$. Since $F(z)=\mathcal{Z}_{\mathcal{G}^{\bullet}}(f(z),f(z^2),...)$ and the cycle index sum relation~\eqref{inventory} holds, $F(z)$ satisfies the implicit equation
\[F(z)=f(z)\exp \bigg( \sum _{k\geq 1}\frac{1}{k}F(z^k)\bigg).\]
Let $\xi(z)=f(z)\exp \big( \sum _{k\geq 2} \frac{1}{k} F(z^k)\big) $ and $\rho,\rho_\xi$ be the smallest positive singularities of $F(z)$ and $\xi(z)$, respectively. Notice that $0<\rho <1$, since it is easy to lower bound and upper bound $[z^n]F(z)$ by an exponential.

We first show that $\xi(z)$ is analytic in $|z|\leq \rho $. The function $f(z)$ has radius of convergence equal to 1, while for $|z|<1$ it holds that
\begin{align*}
\left|\exp \bigg( \sum _{k\geq 2}\frac{1}{k}F(z^k)\bigg)\right| &\leq \exp \bigg( \sum _{k\geq 2}\frac{1}{k}F(|z|^k)\bigg)\leq \exp \bigg( \sum _{k\geq 2}\frac{1}{k}\sum_{n\geq 1} f_n |z|^{nk}\bigg)\\
&=\exp \bigg( \sum _{k\geq 2}\frac{1}{k}\sum_{n\geq 1} f_n |z|^{2n+n(k-2)}\bigg)\\ &\leq \exp \bigg( \sum _{k\geq 2}\frac{1}{k}|z|^{k-2}\sum_{n\geq 1}f_nz^{2n}\bigg)\\ &<\exp \bigg(F(z^2)\sum _{k\geq 0} z^{k}\bigg).
\end{align*}
The radius of convergence of $F(z^2)$ is equal to $\sqrt{\rho _F}>\rho _F$, hence the claim is proved.

Since $\xi (x)$ is analytic at $\rho _F$, we can say $F(z)$ is a solution to the equation $y=G(z,y)$, where $G(z,y)=\exp (y)\xi (z).$ Notice that all requirements of Theorem~\ref{drmota} are satisfied, thus $F(z)$ satisfies the following expansion in a dented domain at $\rho $:
$$F(z)= F_0+F_1(1-z/\rho)^{\frac{1}{2}}+F_2(1-z/\rho)+F_3(1-z/\rho)^{\frac{3}{2}}+\mathcal{O}((1-z/\rho)^4).$$
The function $E(z)^{-1}F(z)$ has a singular expansion at $\rho$ of the same type; to obtain it, it is enough to multiply the regular expansion of $E(z)^{-1}$ at $\rho$ with the singular expansion of $F(z)$ at the same point.
To obtain the singular expansion of $\overline{L}(z)$, we apply the dissymmetry relation~\eqref{alertness} to the singular expansion of $E(z)^{-1}F(z)$. The result is an expansion at $\rho$ with singular exponent $3/2$:
$$\overline{L}(z)=\overline{L}_0+\overline{L}_2(1-z/\rho)+\overline{L}_3(1-z/\rho)^{\frac{3}{2}}+\mathcal{O}((1-z/\rho)^4).$$
 To see concretely that the coefficient of $(1-z/\rho)$ vanishes identically we compute the analytic expression of it, which is equal to
$$-F_1 E(\rho)^{-1}(F_0-1). $$
Then, $F_0$ is identically equal to $1$ since
$$0=1-G_y(\rho ,F(\rho))=1-\exp(F(\rho))\xi (\rho)=1-G(\rho,F(\rho))=1-F(\rho).$$

By Lemma~\ref{lstar}, $\hat{L}(z)$ also has a unique singularity at $\rho$. We can compute the first terms of its singular expansion by writing it in the form
\begin{equation*}
\exp ( \overline{L}(z)-\overline{L}_0 ) \exp \bigg( \overline{L}_0-1+\sum _{k\geq 2}\frac{1}{k}(\overline{L}(z^k)-1)\bigg),
\end{equation*}
and substituting $\overline{L}(z)$  by its singular expansion at $\rho $ and all the second factor by its regular expansion at $\rho $, $K_0+K_2(1-z/\rho)^2+\mathcal{O}((1-z/\rho)^4)$. The coefficient of $(1-z/\rho)^{\frac{3}{2}}$ is equal to
$$\bigg( \overline{L}_3+\overline{L}_0\overline{L}_3+\frac{1}{2}\overline{L}_0^2\overline{L}_3\bigg) K_0.$$
The stated asymptotic forms are obtained by the transfer theorem of singularity analysis that is summarised in Equation \eqref{everybody}.
\end{proof}

\begin{corollary}\label{doublelife}
 The coefficients of $L(z)$ have asymptotic growth of the form:
\[[z^n] L(z) \sim  \frac{c}{\Gamma (-3/2)} \cdot n^{-5/2}\cdot \sqrt{\rho}^{-n} ,\] where $ \rho\approx 0.44074$ ($\sqrt{\rho}^{-1}\approx 1.50628$) and $c\approx 594.24035$ or $c\approx 394.50617$, when $n$ is even or odd, respectively, and $\Gamma$ is the Gamma function.
\end{corollary}
\begin{proof}
Observe that for all $n\in \mathbb{N}\setminus \{0\}$, $[z^{2n}]L(z)=[z^{2n+1}]L(z)$ (recall that $L(z)=\hat{L}(z^2)\frac{1}{1-z}-1$). Thus it is enough to study the even part of $L(z)$, which is $\frac{\hat{L}(z^2)}{1-z^2}-1.$ To do so, we will first look into the behaviour of $\frac{\hat{L}(z)}{1-z}-1.$ This behaves the same as $\hat{L}(z)$ which, by the previous Lemma, satisfies a singular expansion of the form
\begin{equation*}
\hat{L}(z)=\hat{L}_0+\hat{L}_2(1-z/\rho)^2+\hat{L}_3(1-z/\rho)^{\frac{3}{2}}+\mathcal{O}((1-z/\rho)^4)
\end{equation*}
on a dented domain at $\rho$. Then $\frac{\hat{L}(z)}{1-z}-1$ has the same expansion, multiplied by an $\frac{1}{1-\rho}$ factor. Then $\frac{\hat{L}(z^2)}{1-z^2}-1$ satisfies the same type of expansion on $\pm\sqrt{\rho}$. In fact, this expansion can be computed by composing the singular expansion of $\hat{L}(z)\frac{1}{1-z}-1$ at $\rho $ with the regular expansion of $z^2$ at $\sqrt{\rho }$. Then the coefficient of $(1-z/\rho)^{\frac{3}{2}}$ is equal to $2^{\frac{3}{2}}\left(1-\rho\right)^{-1}\hat{L}_3$.

For even $n$, the transfer theorem of singularity analysis implies asymptotic growth of the form
$$ \frac{2^{5/2}\hat{L}_3}{\Gamma (-3/2)n^{5/2}\left(1-\rho\right)}\sqrt{\rho}^n.$$
The odd part of $L(z)$ has the same asymptotic growth, multiplied by $\sqrt{\rho}$.
\end{proof}

\begin{theorem}
\label{elsewhere}
The coefficients of $K(z)$ have asymptotic growth of the form:
\[[z^n]K(z)=c\cdot n^\alpha \cdot \exp (\beta n^{1/2}) (1+O(n^{-1/2})), \] where
$c = \frac{\pi^2}{4} 6^{-5/4}
 \approx 0.26275,\,\,
 \alpha = -7/4,\,\,
  \beta= \sqrt{2\frac{\pi^2}{3}}\approx 2.56509.$
\end{theorem}

\begin{proof} (We thank Carlo Beenakker for pointing out this simplification on our proof.)
We start considering the generating function $G(z)=\frac{1}{(1-z)^2}K(z)$. Observe that $G(z)$ can be written as
$$\prod_{n\geq 1}\frac{1}{(1-z^{2n-1})^2}=\prod_{n\geq 1}\frac{(1-z^{2n})^2}{(1-z^{2n-1})^2 (1-z^{2n})^2}=\prod_{n\geq 1} (1-z^{n})^2 (1+z^n)^2\prod_{m\geq 1}\frac{1}{(1-z^m)^2}=\prod_{n\geq 1}(1+z^n)^2.$$
Hence, $G(z)$ is the generating function for partitions into distinct parts, with 2 types each part (see the sequence A022567 at the OEIS). The asymptotic of this generating function is equal to
$$[z^n] G(z)=  \frac{1}{4}6^{-1/4} n^{-3/4} \exp\left(\sqrt{\frac{2\pi^2}{3} n}\right) (1+O(n^{-1/2})),$$
which is obtained starting from the asymptotic for the number of partitions of $n$ into distinct parts and obtaining a fine estimate for its convolution (see \cite[p.8]{Kotesovec2015} for details).

Now we relate the asymptotic estimates for the coefficients of $G(z)$ with the asymptotic estimates of the coefficients of $K(z)$. Writing $G(z)=\sum_{n\geq 1} p_n z^n$, then from the relation $K(z)=(1-z)^2 G(z)$ we conclude that
$$[z^n] K(z)= p_n-2p_{n-1}+p_{n-2}.$$
We finally obtain the result by using the asymptotic for $p_n$, $p_{n-1}$ and $p_{n-2}$ by expanding them in terms of $n^{-1}$. After cancelation of the main terms, we obtain the claimed statement.\footnote{See the explicit computation in the \texttt{Maple} accompanying session at the website of the first author.}
\end{proof}

\section{Enumeration of link-diagrams}

In this section, we enumerate different kinds of connected link-diagrams (from now on, we refer to them plainly as link-diagrams) that are \emph{rooted}, i.e., an edge is distinguished and ordered. We start with the class of $K_4$-free rooted link-diagrams (Subsection \ref{extirpate}), in which we show the main decomposition technique used in the forthcoming subsections. Later, we deal with the subclass of minimal link-diagrams (Subsection \ref{according}) and link-diagrams arising from the unknot (Subsection \ref{unknot}). The intrinsic difficulty in the enumeration of such subclasses lies in in their overcrossing-undercrossing structure that now has to be taken into account. In Section 6, we develop an argument to obtain asymptotic estimates for the unrooted diagrams.

In the next two sections we will often say that some rooted map $R_1$ is \emph{pasted on an edge} $e$ of another rooted map $R_1$. We will explain now what is meant by that, so that no confusion occurs. First consider that $e$ has a well-defined orientation (let us say given by the alphabetical order of $e$ among the other edges) which determines the orientation by which $M_1$ will be pasted. Then $e$ is subdivided into three parts. The root-edge of $M_1$ is identified with the middle part and the rest of $M_1$ is pasted to the right of the middle part, with respect to the orientation of $e$. Then the middle part is erased. When the new map is seen as a link diagram, we assume that the crossing pattern of the map that was pasted is preserved (as if a connected sum operation was performed).

\subsection{Enumeration of $K_4$-minor-free link-diagrams}\label{extirpate}

We denote by $\mathcal{M} $ the class of $K_4$-minor-free link-diagrams, with size being the number of edges.
Enumerating $\mathcal{M} $ is equivalent to enumerate $K_4$-minor-free 4-regular maps, while taking into account the crossing pattern $\sigma$.
We first give a combinatorial decomposition for the rooted version of $\mathcal{M} $, denoted by $\vv{\mathcal{M}}$, where the root-edge has size zero (recall the definition of rooted maps in Section 2) and the crossing pattern is not taken into account. It is alright to forget temporarily the crossing pattern, since in an object of $\vv{\mathcal{M}}_n$ any crossing pattern of the $2^{n/2}$ possible ones gives a different rooted link-diagram. We will denote by $\vv{\mathcal{M}}^+$ the class $\vv{\mathcal{M}}$ where the root-edge and the crossing pattern is taken into account.

The decomposition is done by adapting the construction of 4-regular graphs in~\cite{noy2019enumeration}. Let us mention that the main simplification compared to \cite{noy2019enumeration} is that in our situation we do not obtain 3-connected components. For completeness, and because this decomposition is critical to understand the following subsections, we write it in full and recall all the needed definitions and arguments.

For a map $R\in\vv{\mathcal{M}}$, where $st$ is the root-edge with initial and final vertex $s$ and $t$, respectively, we write $R^{-}$ for the map $R- st$ 
 (this is called a \emph{network} in map enumeration). Consider the following subclasses of $\mathcal{M}$:
 \begin{enumerate}
\item $\mathcal{U}$ corresponds to maps $R\in\vv{\mathcal{M} }$, where $s=t$ (\emph{loop composition}).
\item $\mathcal{S}$ corresponds to maps $R\in\vv{\mathcal{M} } $, where $R^-$ is connected and has a bridge (\emph{series composition}).
\item $\mathcal{P}$ corresponds to maps $R\in\vv{\mathcal{M} } $, where $R^-$ is 2-edge-connected and either $R^{-}-\{ s,t\}$ is disconnected or $s,t$ are connected with at least three edges in $R$ (\emph{parallel composition}).
\item $\mathcal{F}$ corresponds to maps $R\in\vv{\mathcal{M} }$, where $R^-$ is 2-edge-connected, $R^{-}- \{s,t\}$ is connected, and $s,t$ are connected with exactly 2 edges in $R$.
\end{enumerate}

See Figure~\ref{4-regMaps} for an illustration of these classes, even though the picture will be better understood after Proposition~\ref{principal}.

\begin{figure}[h!]\centering
  \includegraphics[scale=2.2]{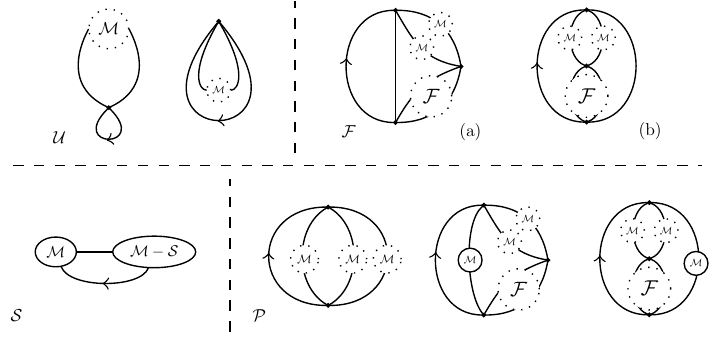}
  \caption{The decomposition of rooted 4-regular maps. In dotted circles, the presence of an object of the inscribed type is optional. In plain circles, it is mandatory.}\label{4-regMaps}
\end{figure}

We denote by $\vv{M}(z)$, or simply $\vv{M}$, the generating function of rooted $K_4$-minor-free link-diagrams, where $z$ marks vertices. Similarly, we denote by $U$, $S$, $P$, and $F$ the corresponding generating functions of the classes $\mathcal{U},\mathcal{S},\mathcal{P},\mathcal{F}$. The following proposition relates all these generating functions in a system of equations:

\begin{proposition}\label{principal}
The generating function of rooted $K_4$-minor-free link-diagrams, $\vv{M}(z)$, satisfies the following system of equations:
\begin{equation*}
\begin{aligned}
\vv{M}  &= U+S+P+F \\
U &= 2z^2\vv{M} +2z\\
S &= z(\vv{M} -S)\vv{M} \\
P &= z^3(1+z\vv{M})^3+zF\vv{M}\\
F &= (z+z^2\vv{M})^2\big(F+2z(z+z^2\vv{M})^2\big)
\end{aligned}
\end{equation*}
\end{proposition}
\begin{proof}(See also~\cite[Lemma 5.1]{noy2019enumeration})
The classes $\mathcal{U},\mathcal{S},\mathcal{P},\mathcal{F}$ are, by definition, disjoint.
Notice that it is also not possible that $R^-$ is disconnected, since this would imply that $st$ is a bridge and would contradict the 4-regularity. Moreover, it is not possible that $R^-$ is 2-edge-connected, $R^{-}- \{s,t\}$ is connected, and $s,t$ are connected with exactly one edge in $R$, since this would force the existence of a $K_4$ minor: First, notice that $s,t$ can have no loop edges, because we assumed 2-edge connectedness in $R^-$. This implies that there are three vertices $s_1,s_2,s_3$ that are neighbours of $s$ and $t_1,t_2,t_3$ that are neighbours of $t$. Also $R^-$ can't have a cut-vertex $v_c$ that leaves $s,t$ in different connected components, since if it did then we could split the connected components in $R^--\{ v_c\}$ in two parts, the ones connected to $s$ and the ones connected to $t$ (the rest can go in either group). But then $v_c$ has two edges in both groups, by 2-edge-connectedness. This implies that in both parts the sum of degrees is odd, a contradiction. Then $ R^-$  is 2-connected and we can find two internally disjoint paths between $s,t$ in $R^-$. The two paths are connected by another path in $R^--\{s,t\}$, since this is assumed connected, and the claim is shown. We conclude that $\vv{\mathcal{M}} $ is partitioned as $\vv{\mathcal{M}}  = \mathcal{U}\cup\mathcal{S}\cup\mathcal{P}\cup\mathcal{F}$.

For $\mathcal{U}$, there are two different maps of size one. Any other map $R\in\mathcal{U}$ can be decomposed uniquely into a map of size one and another map $J$ that is pasted on its non-root edge in the canonical way with respect to the root edge. The latter means that the non-root edge $st$ is subdivided into $svv't$, $vv'$ is removed, and the endpoints of $J$'s  root-edge are identified with $v,v'$, respecting the orientation induced by $R$'s root-edge.

Let  $ R\in \mathcal{S}$. Out of all the bridges in the map $R^-$, we pick the first one with respect to the point of the root-edge. After deleting it, there is a unique submap attached to the point of the root-edge and we draw a new root-edge to it which points where the original root-edge pointed and begins at the vertex the bridge was. We call this map $R_1$, and $R_1^-$ cannot have a bridge by definition. From what is left in $R^-\setminus R_1$, we define similarly $R_2$. Then $R_1\not\in \mathcal{S}$ and $R_2\in\vv{\mathcal{M}}$. The bridge between them is counted by $z$.

If $ R\in \mathcal{F}$, it can be decomposed uniquely to a single edge and a series of double edges, on each of which there may be pasted other maps from $\vv{\mathcal{M}}$, in a canonical way. The factor $(z+z^2\vv{M})^2$ corresponds to the first pair of edges and the factor $F+2z(z+z^2\vv{M})^2$ to the rest of the double-edges. In the latter, the factor $z$ corresponds to the single edge and the factor 2 counts its two possible positions with respect to the root-edge.

For $\mathcal{P}$ there are two cases: either each of the connected components in $R^--\{s,t\}$ is connected with one edge to each of the $s,t$, or there is a component connected with two edges to each of the $s,t$. In the second case we have an object in $\mathcal{F}$, where now an object from $\vv{\mathcal{M}}$ is pasted on the single edge.\end{proof}

  We can now analyze this system of equations by means of asymptotic techniques.

\begin{theorem}\label{collector}
The number of rooted $K_4$-free link-diagrams on $\vv{\mathcal{M}}^+$ with $n$ vertices is asymptotically equal to:
\[[z^n]\vv{M}^+(z)\sim  \frac{c}{\Gamma (-1/2)}\cdot n^{-3/2} \cdot 2^{n/2} \cdot \rho^{-n},\] where $n$ is even and $\rho , c,$ are constants; in particular, $ \rho \approx 0.31184$ ($\frac{2}{\rho }\approx 6.41337$) and $c\approx -3.04531$.
\end{theorem}
\begin{proof}
By Proposition~\ref{principal}, $\vv{M}(z)$ satisfies a polynomial system of equations.
By algebraic elimination we obtain the following polynomial $P_{\vv{M}}(z,y)$, which satisfies that $p_{\vv{M}}(z,\vv M(z))=0$:
\begin{align*}p_{\vv{M}}(z,y) =\; & {y}^{6}{z}^{11}+6\,{y}^{5}{z}^{10}+15\,{y}^{4}{z}^{9}-{y}^{4}{z}^{7}+
20\,{y}^{3}{z}^{8}-4\,{y}^{3}{z}^{6}+15\,{y}^{2}{z}^{7}+{y}^{3}{z}^{4}
-\\
&  -6\,{y}^{2}{z}^{5}+6\,y{z}^{6}+4\,{y}^{2}{z}^{3}-4\,y{z}^{4}+{z}^{5}+5
\,y{z}^{2}-{z}^{3}-y+2\,z .
\end{align*}
We can obtain a similar polynomial $p_{z\vv{M}^+}(z,z \vv M(z))=0$ for $z\vv{M}$.
Setting $F(z,y)=P_{z\vv{M}}(z,y)+y$ the conditions of Theorem~\ref{drmota} hold (even though $F(z,y)$ does not have positive coefficients, it is enough that the original system of equations has; in fact, in~\cite{drmota1997systems} the same statement is shown for such systems with positive coefficients). Hence
$$z\vv{M}(z)=M_0+M_1(1-z/\rho)^{\frac{1}{2}}+M_2(1-z/\rho)+\mathcal{O}((1-z/\rho)^2)$$
in a dented disk around some computable positive number $\rho$. Since $zM(z)$ is periodic, it has a similar singular expansion at $-\rho $ with identical coefficients. By the transfer theorem of singularity analysis,
$$ [z^n]z\vv{M}(z)\sim \frac{M_1}{\Gamma (-1/2)}\rho^{-n}((-1)^n+1).$$
Finally, a factor $2^{n/2}$ accounts for all the possible undercrossings and overcrossings.
\end{proof}

The first terms of the series $\vv{M}^+$ are the following:
\[\vv{M}^+=4\,{z}^{2}+36\,{z}^{4}+432\,{z}^{6}+5984\,{z}^{8}+90112\,{z}^{10}
+1432576\,{z}^{12}+23656960\,{z}^{14}+...
\]

\subsection{Minimal diagrams}
\label{according}

Recall that a link diagram is minimal if, for the link it represents, the number of its edges is the minimum possible.
Let $\mathcal{M}_1$ be the class of all minimal link-diagrams in $\mathcal{M}$, counting by the number of edges.
Let $\vv{\mathcal{M}_1}$ (resp. $\vv{\mathcal{M}^+_1}$) be the rooted version of $\mathcal{M}_1$ where the root-edge is not taken into account (resp. is taken into account). We denote by $\vv{M_1}(z)$ the corresponding generating function (resp. $\vv{M_1^+}(z)$). Here the crossing pattern will be encoded directly in the combinatorial decomposition $\vv{\mathcal{M}_1}$.

In order to assure the minimality condition, one must encode the crossing pattern of each map that is being pasted in the construction of Proposition~\ref{principal}. To this end, we first define the subclasses $\mathcal{M}_1 ,\mathcal{S}_1 , \mathcal{P} _1, \mathcal{F}_1 $ of the classes $\mathcal{M} ,\mathcal{S} , \mathcal{P} , \mathcal{F} $, such that each contains all minimal diagrams of its respective superclass. We then partition each of these classes into four smaller classes: $\mathcal{M} _{i}^j,\mathcal{S} _{i}^j, \mathcal{P} _{i}^j, \mathcal{F} _{i}^j$, where $i,j\in\{-,+\}$. The subscript indicates whether the tail of the root-edge is overcrossing or not and, accordingly, the superscript indicates whether the head of the root-edge is overcrossing or not. See Figure~\ref{root_types} for all possible root-edge types, depending on the overcrossing pattern. We denote by $M_{i}^j,S_{i}^j,P_{i}^j,F_{i}^j$, where $i,j\in\{-,+\}$, the corresponding generating functions (as in the previous Section, all rooted classes are counted by the number of edges, excluding the root-edge). Note that the class $\mathcal{U}$ is not taken into account in this section, since diagrams with a loop can be reduced to diagrams with less crossings by flipping the loop.

\begin{figure}[h!]\centering
  \includegraphics[scale=1.5]{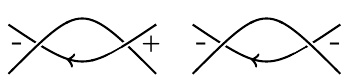}
    \includegraphics[scale=1.5]{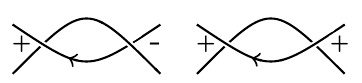}
  \caption{The possible root-edge types.}\label{root_types}
\end{figure}

The decomposition that we will follow here relies strongly on two facts. First that the crossing number of connected sums of torus links is the sum of the crossing numbers of its factors (recall the properties of torus links in the Preliminaries). Second that the operation of pasting a map on an edge, that is essential to the decomposition we saw in the previous section, is equivalent to performing a connected sum operation in the level of links. In our situation, this means that when we paste a minimal link-diagram in an already minimal diagram, then the resulting diagram is again minimal.

\begin{proposition}\label{minimal}
The generating function of minimal, rooted, $K_4$-minor-free link-diagrams $\vv{M_1}(z):=\vv{M_1}$ (where $z$ encodes vertices) satisfies the following polynomial system of equations:

\hspace{-1.2cm}\begin{minipage}{.5\textwidth}
\begin{align*}
\vv{M_1}  &=  M_{+}^{+}+M_{+}^{-}+M_{-}^{+}+M_{-}^{-}\\
M_{-}^{+} &=  S_{-}^{+}+P_{-}^{+}+F_{-}^{+}\\
M_{+}^{-} &= S_{+}^{-}+P_{+}^{-}+F_{+}^{-}\\
M_{+}^{+} &= S_{+}^{+}\\
M_{-}^{-} &=  S_{-}^{-}\\
P_{-}^{+} &=   z^{3}(1+z\vv{M_1})^{3}+zF_-^+\vv{M_1}\\
P_{+}^{-} &=   z^{3}(1+z\vv{M_1})^{3}+zF_+^-\vv{M_1}\\
\end{align*}
\vspace*{-1.3cm}
\end{minipage}
\begin{minipage}{.5\textwidth}
\begin{align*}
S_{+}^{+} &=  z(M_{-}^{+}+M_{+}^{+})(M_{+}^{+}+M_{+}^{-}-S_{+}^{+}-S_{+}^{-})\\[1.5pt]
S_{-}^{-} &=  z(M_{-}^{-}+M_{+}^{-})(M_{-}^{-}+M_{-}^{+}-S_{-}^{-}-S_{-}^{+})\\[1.5pt]
S_{-}^{+} &=  z(M_{+}^{+}+M_{-}^{+})(M_{-}^{-}+M_{-}^{+}-S_{-}^{-}-S_{-}^{+})\\[1.5pt]
S_{+}^{-} &=  z(M_{-}^{-}+M_{+}^{-})(M_{+}^{+}+M_{+}^{-}-S_{+}^{+}-S_{+}^{-})\\[1.5pt]
F_{-}^{+} &=   (z+z^2\vv{M_1})^2(F_{-}^{+}+2z(z+z^2\vv{M_1})^2) \\[1.5pt]
F_{+}^{-} &=   (z+z^2\vv{M_1})^2(F_{+}^{-}+2z(z+z^2\vv{M_1})^2)
\end{align*}
\end{minipage}
\end{proposition}
\begin{proof}
The defining equations for $\vv{M_1},M_{-}^{+},M_{+}^{-}$ are straightforward. Observe that $\mathcal{P}_{+}^{+}, $ $\mathcal{P}_{-}^{-},$ $\mathcal{F}_{+}^{+},$ $\mathcal{F}_{-}^{-}$ are empty, since they can be transformed to diagrams with fewer crossings with a Type II Reidemeister move (in the case of parallel networks, this could require first an ambient isotopy of the link that allows this move). 
Hence, the defining equations for $M_{+}^{+} $ and $M_{-}^{-} $ are also justified. For the classes $S_{i}^{j}$, recall that a series map is decomposed into another map $R_1$ and a non-series map $R_2$, joined together with an edge. Then, the head of its root-edge must agree (with respect to overcrossing or undercrossing) with the head of $R_1$, and the tail of the its root-edge must agree with the tail of $R_2$. This suffices for minimality, since the crossing number in our link classes is additive. In fact, whenever a pasting of an object occurs in this construction, it corresponds to a connected sum and, by  additivity, minimality is not affected. Thus follow the equations for $P_{+}^{-}$ and $P_{-}^{+}$.

Recall that each object in $\mathcal{F}$, thus also in $\mathcal{F}_{i}^{j}$, is associated to a series of double edges. The corresponding crossings are now uniquely defined by $i,j$ and they must alternate. Suppose $R_2\in \mathcal{F}_1$ is used in the recursive construction of $R_1\in \mathcal{F}_{i}^{j}$. Then, there are two cases for $R_2$. Either the crossings of its root edge agree with $i,j$ and it is of the type (b) in Figure~\ref{4-regMaps}, or the crossings of its root edge do not agree with $i,j$ and it is of the type (a). Otherwise, the diagram can be simplified by a Reidemeister Type II move (after a suitable  ambient isotopy of the link that allows this move).  Observe that each such series of $k$ double edges constitutes a minimal link-diagram of the torus link $T(k)$, thus cannot be further simplified. Since the sum of the objects in these two cases is equal to $(\mathcal{F}_i^j)_n$ for every $n$, we can use the GF $F_i^j$. Finally, the objects pasted on the double edges contribute to the crossing number additively.
%
\end{proof}

\begin{theorem}\label{terrorize}
The class of $K_4$-free minimal rooted link-diagrams $\vv{\mathcal{M}^+_1}$ grows asymptotically as:
\[[z^n] \vv{\mathcal{M}_1}(z)\sim  \frac{c}{\Gamma (-1/2)}\cdot n^{-3/2} \cdot \rho ^{-n} ,\] where $n$ is even and $\rho , c,$ are constants; in particular, $ \rho \approx  0.41456$ ($\rho ^{-1}\approx 2.41214$) and $c\approx -1.62846$.
\end{theorem}
\begin{proof}
The proof is almost identical to the one in Theorem~\ref{collector}. Only two things change: the defining polynomial of $\vv{\mathcal{M}_1}(z)$,  is equal to
\begin{align*}p_{\vv{M}_1}(z,y) =\; &  2\,{y}^{6}{z}^{11}+12\,{y}^{5}{z}^{10}+30\,{y}^{4}{z}^{9}+2\,{y}^{4}{z
}^{7}+40\,{y}^{3}{z}^{8}+8\,{y}^{3}{z}^{6}+30\,{y}^{2}{z}^{7}+{y}^{3}{
z}^{4}+\\ &  +12\,{y}^{2}{z}^{5}+12\,y{z}^{6}  +2\,{y}^{2}{z}^{3}+8\,y{z}^{4}+2
\,{z}^{5}+y{z}^{2}+2\,{z}^{3}-y.
\end{align*}
and the crossing pattern is already taken into account by the combinatorial construction.
\end{proof}

The first terms of the series $\vv{M_1}^+$ are the following:
\[\vv{M_1}^+=2\,{z}^{4}+4\,{z}^{6}+20\,{z}^{8}+84\,{z}^{10}+372\,{z}^{12}+1796\,{z
}^{14}+8516\,{z}^{16}+42340\,{z}^{18}+211332\,{z}^{20}+...
\]

\subsection{Link-diagrams of the unknot}\label{unknot}

Let $\vv{ \mathcal{M}_2}, \vv{ \mathcal{M}_2^+}$ be the classes of rooted link-diagrams of the unknot not counting and counting, respectively, the root-edge. Let $\mathcal{M}_2$ be the unrooted $\vv{ \mathcal{M}_2^+}$. We define the $\vv{ \mathcal{M}_2}$ subclasses $\mathcal{U}_2, \mathcal{S}_2 , \mathcal{P} _2, \mathcal{F}_2 $ of $ \mathcal{U},\mathcal{S} , \mathcal{P} , \mathcal{F}$, such that each contains all diagrams of the unknot in its respective superclass. We then partition each of these classes into four smaller combinatorial classes, which we denote with the same symbols as in the previous subsection for simplicity, i.e., $\mathcal{M} _{i}^j,\mathcal{U}_i^j,\mathcal{S} _{i}^j, \mathcal{P} _{i}^j, \mathcal{F} _{i}^j$, where $i,j\in\{-,+\}$. We denote by $M_{i}^j,U_i^j,S_{i}^j,P_{i}^j,F_{i}^j$, where $i,j\in\{-,+\}$, the corresponding generating functions (keeping the same convention, all rooted classes are counted by the number of edges, excluding the root-edge).

We also need the classes $\mathcal{T}_{r}$, $r\in \{ 1, 3\}$, that correspond to all possible ways to split a sequence of $2n+r$ points into two groups of size $n$ and $n+r$. Then,
\[ T_r(z) =z^r\sum _{n\geq 0} \binom{2n+r}{n}z^{2n}.\]
Observe that
\begin{align*}
  \sum _{n\geq 0} \binom{2n+r}{n}z^{n} &=  1+\sum _{n\geq 1}\bigg(\binom{2n+r-1}{n-1}+ \binom{2n+r-1}{n} \bigg) z^{n} \\
  &=  \sum _{n\geq 1}\binom{2n+r-1}{n-1} z^{n}+ \sum _{n\geq 0}\binom{2n+r-1}{n} z^{n}  \\
  &=  \frac{2z}{r}\sum _{n\geq 0}\frac{nr}{2n+r}\binom{2n+r}{n}z^{n-1}+\sum _{n\geq 0}\frac{r}{2n+r}\binom{2n+r}{n}z^n\\
    &= \frac{2z}{r}[B_2(z)^{r}]'+B_2(z)^{r} ,
 \end{align*}
where $B_t(z)$ are known as \emph{generalised binomial series}, which satisfy the equality $$B_t(z)^{r}=\sum _{n\geq 0}\binom{tn+r}{n}\frac{r}{tn+r}z^n$$ (see \cite[Ch. 5.4]{knuth1989concrete}). In particular, $B_2(z)$ is the series of Catalan numbers: $B_2(z)=\frac{1-\sqrt{1-4z}}{2z}$. Then
\[ T_r(z) =z^r\bigg[\frac{2z}{r}[B_2(z)^{r}]'+B_2(z)^{r}\bigg]\bigg| _{z=z^2}.\]

%


\begin{proposition}
 The generating function of $K_4$-minor-free rooted link-diagrams of the unknot, $\vv{M_2}(z)$, denoted also by $\vv{M_2}$, satisfies the following polynomial system of equations:

\hspace{-1.6cm}\begin{minipage}{.5\textwidth}
\begin{align*}
\vv{M_2}  &=  M_{+}^{+}+M_{-}^{+}+M_{+}^{-}+M_{-}^{-}\\
M_{+}^{+} &=  S_{+}^{+}+P_{+}^{+}+F_{+}^{+}\\
M_{-}^{+} &=  S_{-}^{+}+P_{-}^{+}+F_{-}^{+}+L^+\\
M_{+}^{-} &=  S_{+}^{-}+P_{+}^{-}+F_{+}^{-}+L^-\\
M_{-}^{-} &=  S_{-}^{-}+P_{-}^{-}+F_{-}^{-}\\
P_{+}^{+} &=   zF_+^+\vv{M_2}\\
P_{-}^{+} &=   zF_-^+\vv{M_2}\\
P_{+}^{-} &=   zF_+^-\vv{M_2}\\
P_{-}^{-} &=   zF_-^-\vv{M_2}\\
\end{align*}
\vspace*{-1.3cm}
\end{minipage}
\hspace{-.9cm}\scalebox{.93}{\begin{minipage}{.5\textwidth}
\begin{align*}
U^+ &=  2z+z^2\vv{M_2}\\
U^- &=  2z+z^2\vv{M_2}\\
S_{+}^{+} &=  z(M_{-}^{+}+M_{+}^{+})(M_{+}^{+}+M_{+}^{-}-S_{+}^{+}-S_{+}^{-})\\
S_{-}^{+} &=  z(M_{+}^{+}+M_{-}^{+})(M_{-}^{-}+M_{-}^{+}-S_{-}^{-}-S_{-}^{+})\\
S_{+}^{-} &=  z(M_{-}^{-}+M_{+}^{-})(M_{+}^{+}+M_{+}^{-}-S_{+}^{+}-S_{+}^{-})\\
S_{-}^{-} &=  z(M_{-}^{-}+M_{+}^{-})(M_{-}^{-}+M_{-}^{+}-S_{-}^{-}-S_{-}^{+})\\
F_{-}^{+} &=   4z(z+z^2\vv{M_2})^2T_1((z+z^2\vv{M_2})^2) \\
F_{+}^{-} &=   4z(z+z^2\vv{M_2})^2 T_1((z+z^2\vv{M_2})^2)\\
F_{+}^{+} &=  2z(z+z^2\vv{M_2})^2\big(T_{1}((z+z^2\vv{M_2})^2)+T_{3}((z+z^2\vv{M_2})^2)\big) \\
F_{-}^{-} &=  2z(z+z^2\vv{M_2})^2\big( T_{1}((z+z^2\vv{M_2})^2)+T_{3}((z+z^2\vv{M_2})^2)\big)
\end{align*}
\end{minipage}}
\end{proposition}
\begin{proof}
The defining equations for $\vv{M_2}, M^i_j,S^i_j$ can be justified in the same way as in Proposition~\ref{principal} and Proposition~\ref{minimal}. Let $R\in \mathcal{P}_i^j$. If $R^--\{s,t\}$ is empty or disconnected, then $R$ has two components and does not represent the unknot. Thus, $P_{i}^{j}  =   zF_i^j\vv{M_2}$ (recall the construction in Proposition~\ref{principal}).

The equations for the classes $\mathcal{F}_i^j$ need to change substantially. Let $R\in \mathcal{F}_i^j$. Recall that $R$ is decomposed into the root-edge $e_1$, an edge $e_2$ parallel to it (either to the left or to the right face that is adjacent to $e_1$), and a chain of double edges, $C$, on which other objects of $\vv{\mathcal{M}_2}$ may be pasted.

Traversing the knot in the direction of the root-edge, we can associate on each point of the knot a tangent arrow. Consider the corresponding arrows on the link-diagram and notice that each crossing point has two such arrows. Moreover, there is a unique face of the diagram that is adjacent to both arrows, let us call it $F$. On each crossing point, we associate a plus sign or a minus sign, according to whether the left or the right arrow is overcrossing, with respect to the joint direction of the two arrow heads on $F$. Observe that if two consecutive vertices on $C$ bear different signs, the diagram can be reduced by a move of Type II. Hence, in order to obtain a trivial knot, the sum $s$ of the signs should be either $+1$ or $-1$: otherwise, either we have more than one component, or the diagram corresponds to a non-trivial knot. The sum of the signs of the root vertices can be $0$ or $\pm 2 $.

We use the generating functions $T_{1}$ and $T_{3}$ that encode all the possibilities, so that the total sum of the signs on $C$ equals $\pm1$. In particular, when the sum on the root vertices is zero, we use the GF $T_1$ twice, since we distinguish on whether the total sum is $-1$ or $1$. When the sum of the root vertices is $2$ or $-2$, we use the functions $T_1,T_3$ that account likewise for both cases. We substitute each atom on $T_1,T_3$ by a double edge that may or may not have other objects pasted, and obtain $T_{1}((z+z^2\vv{M_2})^2),T_{3}((z+z^2\vv{M_2})^2)$. The extra factor $(z+z^2\vv{M_2})^2$ accounts for the first  double edge after the head of the root.
%
%
%
\end{proof}


\begin{theorem}\label{necessary}
The class of $K_4$-free rooted link-diagrams of the unknot $\vv{\mathcal{M}_2^+}$ grows asymptotically as:
\[[z^n]\vv{M_2}(z)\sim \frac{c}{\Gamma (-1/2)}\cdot n^{-3/2} \cdot \rho ^{-n} ,\] where $n$ is even and $\rho , c,$ are constants; in particular,  $ \rho \approx 0.23626$ ($\rho ^{-1}\approx 4.23249$), $c\approx -3.39943$.\end{theorem}
\begin{proof}
We first obtain the defining polynomial of $\vv{M_2}(z)$ with respect to $z,M, t_1$,$t_3$, denoted by $p_{\vv{M_2}}$, by means of algebraic elimination ($y$ stands for $M$):
\begin{align*}p_{\vv{M}_2}(z,y, t_1,t_3) =\; & 12\,y^{4}{z}^{7}t_1+4\,y^{4}{z}^{7}t_3+48\,y^{3}{z}^{6}t
_{{1}} +16\,y^{3}{z}^{6}t_3+72\,y^{2}{z}^{5}t_1+\\ & +24\,y^{2}
{z}^{5}t_3+48\,y^{4}t_1+16\,y^{4}t_3+4\,y^{2}{z}^{
3}+12\,{z}^{3}t_1+4\,{z}^{3}t_3+\\  &+8\,y^{2}-y+4\,z.
\end{align*}
Then, we substitute $t_1$ and $t_3$ by the closed forms of $T_1(z)$ and $T_3(z)$, where $z$ is substituted by $(z+z^2x)^2$. The rest of the analysis is identical to Theorem~\ref{collector}, noting that in this case the crossing pattern is already taken into account by the combinatorial construction.
\end{proof}

\noindent The first terms of the series $\vv{M_2}^+$ are the following:
\[\vv{M_2^+}=4\,{z}^{2}+32\,{z}^{4}+332\,{z}^{6}+3968\,{z}^{8}+51688\,{z}^{10}+
712416\,{z}^{12}+10214604\,{z}^{14}+150776064\,{z}^{16}+...\]

\section{The unrooting argument}\label{unrooting}

In this section, we develop an unrooting argument for the families of link-diagrams we have enumerated, using results from~\cite{richmond1995almost} and~\cite{bender1992submaps}.

A link-diagram is \emph{symmetric} if it admits a non-trivial map automorphism that also identifies the crossing structure $\sigma$. We will first show that the proportion of objects in $\mathcal{M}_n,$ $(\mathcal{M}_1)_n,$ $(\mathcal{M}_2)_n$ that are symmetric is exponentially small. To do so, it is enough to show the statement in the map level. Then, we can deduce asymptotic estimates for $|\mathcal{M}_n|,$ $|(\mathcal{M}_1)_n|,$ $ |(\mathcal{M}_2)_n|$ from the estimates we already have for $|\vv{\mathcal{M}}^+_n|,$ $|(\vv{\mathcal{M}}_1^+)_n|,$ $ |(\vv{\mathcal{M}}_2^+)_n|$.

We state some definitions from~\cite{richmond1995almost}: a \emph{submap} $R'$ of a map $R$ is a map such that $R'$ is a set of faces of $R$ and their boundary edges and vertices, and $R'$ is continuous.
We call $R\setminus R'$ the map obtained after removing the faces of $R'$.
 We say that two maps are \emph{glued} when we identify their outer faces, which have the same degree.

A map $R'$ is called \emph{outercyclic} if the edges of its unbounded face induce a cycle with no repeated vertices. It is called \emph{free} if in all its occurrences as submap in maps $R$, all maps resulting by gluing $R'$ to $R\setminus R'$, on the face where $R'$ initially belonged (let us call this face of $R\backslash R'$, $R$-\emph{hole}), belong to the same class of maps as R. It is called \emph{ubiquitous} if for small enough $c>0$, there is a positive $d<1$ such that the proportion of objects in $\mathcal{R}$ that do not contain at least $cn$ copies of $R'$ is at most $d^n$ for large enough $n$. Two maps have disjoint appearances when they do not share a face. The main Theorem in~\cite{richmond1995almost} gives sufficient conditions for a rooted map class to contain exponentially few symmetric maps.
\begin{theorem}[\cite{richmond1995almost}] Let $\mathcal{C}$ be a class of rooted maps on a surface. Suppose that there is an outer-cyclic rooted planar map $R_1$ such that in all maps in $\mathcal{C}$, all copies of $R_1$ are pairwise disjoint, and such that $R_1$
\begin{enumerate}
\item has no reflection symmetry in the plane (i.e. reflective symmetry preserving the unbounded face);
\item is free and ubiquitous in $\mathcal{C}$.
\end{enumerate}
Then the proportion of $n$-edged maps in $\mathcal{C}$ with non-trivial automorphisms is exponentially small.\label{almostall}
\end{theorem}

We will use the map depicted in Figure~\ref{anspricht}. When working with the classes $\mathcal{M},\mathcal{M}_1$ we can consider that it has the crossing pattern of $R_A$. When we work with $\mathcal{M}_2$ we can consider that it has the crossing pattern of $R_B$.
Notice that the map has size equal to 39.
\begin{figure}[h!]\centering
  \includegraphics[scale=1.5]{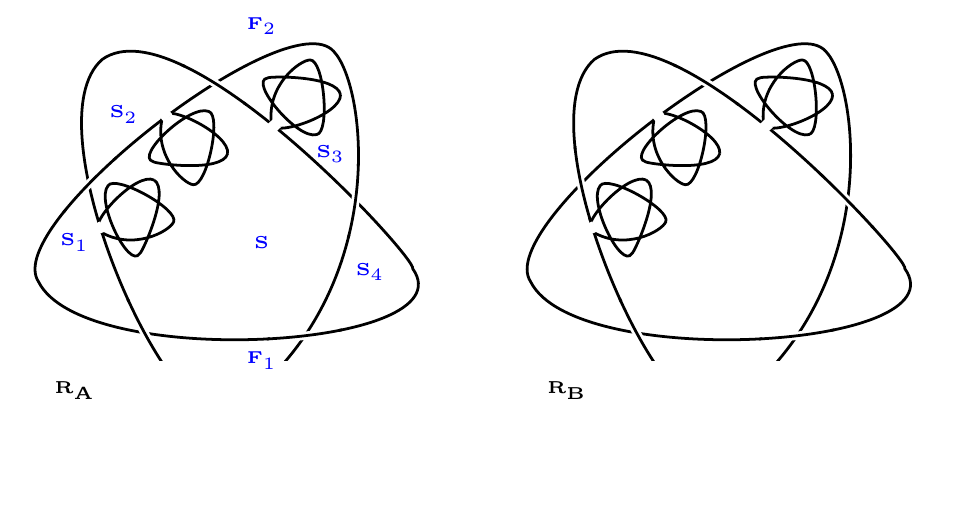}\vspace*{-1.8cm}
  \caption{An asymmetric and ubiquitous link-diagram in $\protect\vv{\mathcal{M}^+}$ and $\protect\vv{\mathcal{M}_1^+}$ ($R_A$), and in $ \protect\vv{\mathcal{M}_2^+}$ ($R_B$). The missing crossings follow the pattern of the existing crossings.}\label{anspricht}
\end{figure}
\begin{lemma}\label{corrected}
The appearances of $R_A$ as submap in $\vv{\mathcal{M}^+}$ or $\vv{\mathcal{M}_1^+}$ are disjoint. The same holds for the appearances of $R_B$ in $\vv{\mathcal{M}_2^+}$.
\end{lemma}
\begin{proof}
It is enough to prove the claim for $R_A$ and $\vv{\mathcal{M}^+}$. We denote by $S,S_i,F_i$ faces and neighbouring faces of $R_A$, respectively, as shown in Figure~\ref{anspricht}.
Let $R\in \vv{\mathcal{M}^+}$ and two distinct submaps of $R$, called $R_A^1,R_A^2$, such that $R_A^1, R_A^2\cong R_A$. Then, there is a map isomorphism $\phi$ that identifies $R_A^1$ to $R_A^2$.

Assume that $R_A^1,R_A^2$ share a face. If $S\in R_A^2$, then $R_A^2=R_A^1$, so we can assume that $S\not\in R_A^2$. Then at least one of the remaining border faces $S_i$ must belong to $R_A^2$, which implies that $F_2$ is equivalent to $S$ under $\phi$. But this is impossible regardless of $F_1$, since $F_2$ is neighbouring with one edge to at least four faces, while $S$ is neighbouring with one edge to exactly two faces ($S_4$ and $F_1$).\end{proof}


Lemma~\ref{brightest} follows by the general result~\cite[Theorem 2]{bender1992submaps}. For the sake of clarity and to point out that~\cite[Theorem 2]{bender1992submaps} also holds for maps with a crossing structure, we reproduce here the complete argument in a simplified way for our maps.

\begin{lemma}\label{brightest}
There exists $c>0$ small enough such that the proportion of objects in $\vv{\mathcal{M}_n^+},$ $(\vv{\mathcal{M}_1^+})_n,$  (resp. $(\vv{\mathcal{M}_2^+})_n$) that do not contain at least $cn$ copies of $R_A$ (resp. $R_B$) is exponentially small.\end{lemma}
\begin{proof}
Let $\mathcal{H}$ be the class of objects in $\vv{\mathcal{M}}$ that contain less than $cn$ copies of $R_A$, where $c\leq \frac{1}{2}$, and $H(z)$ the corresponding generating function. Let $\mathcal{Q}$ be the class such that $Q\in \mathcal{Q}$ is constructed from a map in $\mathcal{H}$, where on each non-root edge one pastes or not $R_A$ in the canonical way (repetitions are allowed, that is, there might be several copies of $Q$ in $\mathcal{Q}$). Let $Q(z)=\sum _{n\geq 0}q_nz^n$ be the corresponding generating function, where size is again taken with respect to edges, not counting the root. Then $Q(z)=H(z+z^{40})$ (it is 40 and not 39, by the way we have defined the operation of pasting a map on an edge, in Section 5). Notice that all $Q\in\mathcal{Q}$ belong to $\mathcal{M}$, since the operation of pasting a map on an edge corresponds to a connected sum between the represented links.

Denote by $r(\cdot )$ the radius of convergence of some generating function. We will use the following lemma, proved in~\cite[Lemma 2]{bender1992submaps}.

\vspace{.3cm}
\textit{Lemma I:} If $F(z)\neq 0$ is a polynomial with non-negative coefficients and $F(0)=0$, $H(z)$ has a power series expansion with non-negative coefficients and $0<r(H)<\infty$, and $Q(z)=H(F(z))$, then $r(H)=F(r(Q))$.
\vspace{.3cm}

Since by Lemma~\ref{corrected} the copies of $R_A$ are always disjoint, every object in $\mathcal{H}$ can be repeated at most \[\sum _{k\leq cn}\binom{n}{cn}\leq \bigg(\frac{e}{c}\bigg)^{cn}\]
times. Let us call this number $t_n$. Then,
\[
\frac{1}{r(\vv{M})}\geq\limsup _{n\rightarrow \infty}\bigg(\frac{q_n}{t_n}\bigg)^{1/n}\geq \frac{\big(\frac{c}{e}\big)^{c}}{r(Q)}
\]
By Lemma I, it holds that $r(H)=r(Q)(1+r(Q)^{40})$. Consequently,
\[\frac{r(H)}{r(\vv{M})}\geq (1+r(Q)^{40}) \bigg(\frac{c}{e}\bigg)^c>1 \] for small enough $c$ and the statement follows for $\vv{\mathcal{M}}$.
Then it immediately follows for $\vv{\mathcal{M}^+}$ and the cases of $\vv{\mathcal{M}_1^+},\vv{\mathcal{M}_2^+}$ can be treated similarly.
\end{proof}
\noindent Notice that the map $R_A$ cannot satisfy freeness in any non-trivial link-diagram family, because of 4-regularity. The same would hold for any other map. Hence, to prove the final statement we need a relaxed version of Theorem~\ref{almostall}.
\begin{lemma}\label{relax}
Theorem~\ref{almostall} holds under a relaxed freeness condition, namely that $R_1$ can be glued to a hole where it initially belonged exactly $j>1$ times, for some constant $j$.
\end{lemma}
\begin{proof}
We will sketch the proof of Theorem~\ref{almostall}, adapted in the stated relaxed condition.

Let $\mathcal{C}_n$ be the class of $n$-sized rooted maps of $\mathcal{C}$ that contain at least $cn$ copies of $R_1$. For $M\in C_n$, let $H(M)$ (or just $H$ when appropriate) be the map obtained by cutting out all the copies of $R_1$ (apart from the case where the copy contains the root internally) and by $C(H)$ the set of maps obtained after pasting them back in all $j$ available ways. (Notice that in our regular maps that contain crossing information, $H$ might not belong to the original class. However, all elements in $C(H)$ belong there.) Then $|C(H)|=j^{k(H)}$, where $k(H)$ is the number of $R_1$-holes in $H$. Let $C$ be the set of all such rooted maps $H$ and, for a rooted map $R$, denote by $\hat{R}$ its unrooted version. For an automorphism of $\hat{M}\in C(H)$, denote by $h(\sigma)$ the automorphism restricted in $\hat{H}$, and by $m(\tau )$, where $\tau$ is an automorphism of $\hat{H}$, the set of maps $\hat{M}\in C(H)$ that admit an automorphism $\sigma$ such that $h(\sigma)=\tau$. Observe that any such $\tau$ has at most two fixed points, hence there are at most $k(H)/2+1$ orbits in total. Then a map in $m(\tau)$ is uniquely determined by the way $R_1$ is pasted in $H$, in one face from each of these orbits. This implies that $|m(\tau )|\leq j^{k(H)/2+1}$. Let $S(\mathcal{C}_n)$ be the number of maps in $\mathcal{C}_n$ that have a non-trivial automorphism and $H(\mathcal{C}_n)=\{H(M)|M\in \mathcal{C}_n\}$. Then the following holds:
\begin{align*}
 S(C_n)&\leq \sum\limits_{\substack{H\in H(\mathcal{C}_n) \\ \tau \in\mathrm{Aut}( \hat{H}) }}m(\tau) \leq  \sum\limits_{\substack{H\in H(\mathcal{C}_n) \\ \tau \in\mathrm{Aut}( \hat{H}) }}j^{k(H)/2+1}=\sum\limits_{\substack{H\in H(\mathcal{C}_n) \\ \tau \in\mathrm{Aut}( \hat{H}) }}\frac{|C(H)|}{j^{k(H)/2-1}}\\
&\leq \sum\limits_{\substack{H\in H(\mathcal{C}_n)}}\frac{4n|C(H)|}{j^{k(H)/2-1}}\leq \sum\limits_{\substack{H\in C}}\frac{4n|C(H)|}{j^{cn/2-1}}=\frac{4nC_n}{j^{cn/2-1}},
\end{align*}
where the last equality follows by the disjointness of the sets $C(H)$. Hence, \[\frac{S(C_n)}{C_n}\leq \frac{4n}{j^{cn/2-1}}.\]\end{proof}

\begin{theorem}\label{dissolved}
The proportion of objects in $\mathcal{M}_n,$ $(\mathcal{M}_1)_n,$ $(\mathcal{M}_2)_n$ that is symmetric is exponentially small.
\end{theorem}
\begin{proof}
The maps $R_A,R_B$ have no reflective symmetry in the plane.
Morever, they always appear disjointly in their respective classes, by Lemma~\ref{corrected}, and they are ubiquitous, by Lemma~\ref{brightest}. The only condition that is missing to apply Theorem~\ref{almostall} is freeness, since in our case there are exactly two distinct gluings of them: the identity and the reflection around the vertical axis, because of the 4-regularity. By Lemma~\ref{relax}, we can relax the freeness condition and conclude the proof.
\end{proof}

\noindent We can now get the final asymptotic result for the link-diagram families under study:

\begin{corollary}\label{final-asymptotic}
The class of connected $K_4$-free link-diagrams $\mathcal{M}$ satisfies (for $n$ even):
$$[z^n]M(z)\sim  \frac{1}{2n}\cdot\frac{cn^{-3/2}}{\Gamma (-1/2)}\cdot\rho ^{-n}\cdot2^{n/2},\,\, \rho \approx 0.31184,\, c_1\approx -3.04531$$
The class of connected  $K_4$-free minimal link-diagrams, $\mathcal{M}_1$, and the class of $K_4$-free link-diagrams of the unknot, $\mathcal{M}_2$, satisfy (for $n$ even):
$$[z^n]M_1(z)\sim \frac{1}{2n}\cdot \frac{c_1n^{-3/2}}{\Gamma (-1/2)}\cdot\rho_1 ^{-n},\,\,  \rho_1 \approx 0.41456,\, c_1\approx -1.62846,$$
$$[z^n]M_2(z)\sim  \frac{1}{2n}\cdot\frac{c_2n^{-3/2}}{\Gamma (-1/2)}\cdot\rho_2 ^{-n},\,\,  \rho_2 \approx 0.23626,\, c_2\approx -3.39943.$$
\end{corollary}

\section{Open problems}

In this paper we made a first step in the enumeration of knot diagrams, starting with $K_{4}$-minor free graphs. Some possible directions for further research are the following.
\begin{itemize}
\item  In Subsection~\ref{unknot} we enumerated the $K_{4}$-minor free
link-diagrams of the unknot. It is interesting to extend this direction to other types of knots, further than  the unknot,  such as the $(2,q)$-torus link for different values of  $q$.
\item It is an open challenge to go further than the $K_{4}$-minor-free link-diagrams. A first candidate in this direction is to consider graph classes where other minors are excluded,
such as planar graphs of bounded treewidth ($K_{4}$-minor free graphs are exactly
those of treewidth at most 2). A first step in this direction is to look for a structural
characterization of the $4$-regular planar graphs with treewidth at most three (in analogy to Theorem~\ref{acceptance}).
\end{itemize}

\noindent{\bf Acknowledgement:} The authors thank the anonymous referees for a detailed reading of the manuscript and for providing useful comments, as well as pointing out minor oversights and improvements in the presentation. We also thank Carlo Beenakker and the MathOverflow community for pointing out an important simplification in the proof of Theorem \ref{elsewhere}. The second author wishes to thank \href{https://sidiropo.people.uic.edu}{Anastasios Sidiropoulos} for early discussions on knots, during \href{https://www.mfo.de/occasion/1542}{MFO Workshop 1542 on Computational Geometric and Algebraic Topology}, that inspired this research.


\end{document}